\documentclass[11pt]{amsart}

\usepackage[margin=1.1in]{geometry}
\usepackage{xcolor}
\usepackage{amssymb, amsmath, amsthm, amsfonts, mathtools}
\usepackage[all]{xy}
\usepackage{enumerate}
\usepackage{mathrsfs}
\usepackage{needspace}

\usepackage{wrapfig}
\usepackage{bbm}
\usepackage{parskip}

\usepackage{hyperref}
\usepackage[nameinlink,capitalise,noabbrev]{cleveref}

  \definecolor{dark-red}{rgb}{0.6,0.15,0.15}
   \definecolor{dark-blue}{rgb}{0.15,0.15,0.6}
   \definecolor{medium-blue}{rgb}{0,0,0.5}

\hypersetup{
    colorlinks, 
    linkcolor=dark-red,
    citecolor=dark-blue, urlcolor=medium-blue
}

\usepackage{tikz}
\usepackage{tikz-cd}

\usetikzlibrary{
arrows.meta,
calc,
decorations.pathmorphing,
matrix,
positioning
}

\makeatletter
\def\l@subsection{\@tocline{2}{-7pt}{2.75pc}{5pc}{\small}}
\makeatother

\usepackage{thm-restate}

\newtheorem{thm}{Theorem}[section]
\newtheorem{lem}[thm]{Lemma}
\newtheorem{prop}[thm]{Proposition}
\newtheorem{cor}[thm]{Corollary}
\newtheorem{defn}[thm]{Definition}
\newtheorem{exmp}[thm]{Example}
\newtheorem{remark}[thm]{Remark}

\newcommand{\Ext}{\mathrm{Ext}}
\newcommand{\Tor}{\mathrm{Tor}}
\newcommand{\Tot}{\mathrm{Tot}}

\DeclareMathOperator{\Hom}{Hom}

\DeclareMathOperator{\coker}{coker}

\title{
Poincar\'e Duality and Quadratic Refinements over Laurent Rings 
}
\author{B\l a\.zej Ruba}
\address{University of Warsaw}
\email{bruba@fuw.edu.pl}

\author{Bowen Yang}
\address{Harvard University}
\email{byyang@alumni.caltech.edu}
\date{\today}

\begin{document}

\begin{abstract}
We develop a Poincaré duality theory for defects of nondegenerate sesquilinear pairings over Laurent polynomial rings. A key ingredient is a novel flat resolution of the character module, constructed from a triangulation of the sphere at infinity associated with a fan. The cup product on this resolution turns Poincaré duality on the sphere into canonical pairings between the resulting defect modules. In middle degrees, we construct distinguished quadratic refinements using equivariant cohomology of the sphere with the antipodal action. The effective replacement of the sphere with a projective space provides a geometric substitute for division by two. Applied to translation-invariant Pauli stabilizer codes, our results establish the nondegeneracy of higher-dimensional braiding pairings. They extend the two-dimensional T-junction formula for topological spin to higher dimensions, while giving it a geometric interpretation.
\end{abstract}
\maketitle
\setcounter{tocdepth}{2}

\vspace*{-1.25\baselineskip}
\tableofcontents

\setlength{\parskip}{0.5em}

\section{Introduction}

In this paper, we study finitely generated modules over Laurent polynomial rings equipped with sesquilinear pairings. The modules are not assumed to be projective, and the pairings, while nondegenerate, need not be unimodular. The failure of unimodularity and projectivity is described by homological obstructions, which we call the defects of the pairing. We develop a~Poincar\'e duality theory for these defects, culminating in the construction of canonical quadratic refinements for middle-degree pairings. 
 
This framework is motivated by condensed matter physics, where duality defects describe excitations and symmetries in quantum lattice systems known as Pauli stabilizer codes. However, the resulting interplay between algebra and geometry is of independent mathematical interest. We~formulate the setup and our main results in Section \ref{sec:results}, and~discuss the motivation from theoretical physics in Section \ref{sec:physics}. Further applications of our results to Pauli stabilizer codes will be explored in a forthcoming paper.

\subsection{Main results} \label{sec:results}
Fix integers $n\geq 2$ and $d \geq 1$, let $k=\mathbb Z/n\mathbb Z$, and consider the Laurent polynomial ring
\begin{equation}
R=k[\mathbb Z^d]=k[x_1^{\pm1},\ldots,x_d^{\pm1}]
\end{equation}
with the involution $\overline{x^\lambda}=x^{-\lambda}$. In this article we study finitely generated $R$-modules equipped with nondegenerate sesquilinear pairings, not necessarily unimodular. Thus, if
\begin{equation}
\Omega\colon M\times N\longrightarrow R
\end{equation}
is nondegenerate, the induced maps $N\to M^*$ and $M\to N^*$ are injective but need not be surjective. Their cokernels and higher derived analogues measure the failure of the pairing to be unimodular and of $M,N$ to be projective.

Our main results are a Poincaré duality theorem for these duality defects, and, in two natural specialized settings, a construction of quadratic refinements of self-pairings in the middle homological degree. Pairings between duality defects are constructed in terms of the cup product of cochains on the sphere \( S^{d-1} \), which we view as the boundary of \( \mathbb R^d \) at infinity. The construction of quadratic forms also involves the geometry of the real projective space \( \mathbb{RP}^{d-1} \).

The basic algebraic object behind the construction is the character module of $R$. We write
\begin{equation}
\widehat R
= k[[x_1^{\pm 1},\dots,x_d^{\pm 1}]] =
\left\{
\sum_{\lambda \in \mathbb Z^d} a_\lambda x^\lambda\ \middle|\ a_\lambda \in k
\right\}
\end{equation}
for the $R$-module of unrestricted formal Laurent series. It is an injective module, naturally identified with the character module $\Hom_k(R, k)$ of $R$. The first part of the paper gives a~finite flat resolution whose combinatorial structure is governed by a cellular structure of \( S^{d-1} \). The geometry of the sphere is relevant for the algebraic structure of $\widehat R$ because it parametrizes the asymptotic directions in which formal power series may extend. 

Let $\Lambda=\mathbb Z^d$ and $\Lambda_{\mathbb R}=\Lambda\otimes_{\mathbb Z}\mathbb R$. We consider its sphere at infinity
\begin{equation}
    S_\infty(\Lambda_{\mathbb R})=(\Lambda_{\mathbb R}\setminus\{0\})/\mathbb R_{>0} \cong S^{d-1}.
\end{equation}
A complete rational polyhedral fan $\Sigma$ gives a~cellular decomposition of this sphere. For every cone $\sigma\in\Sigma$, we introduce an $R$-submodule $\mathcal R(\sigma)\subset\widehat R$ generated by series supported in $\sigma$. If $\sigma$ is strongly convex, $\mathcal R(\sigma)$ is in fact a flat $R$-algebra. Putting these modules on the cells of $S^{d-1}$ gives a cellular cochain complex $C^\bullet(\Sigma;\mathcal R)$. Our first main result is the following.

\begin{restatable}{maintheorem}{Resolution}
\label{thm:cellular}
Let $\Sigma$ be a~complete rational polyhedral fan in $\Lambda_{\mathbb R}$. The sequence
\begin{equation}\label{eq:augmented_exact_complex}
0\longrightarrow R
\longrightarrow C^0(\Sigma;\mathcal R)
\longrightarrow \cdots
\longrightarrow C^{d-1}(\Sigma;\mathcal R)
\longrightarrow \widehat R
\longrightarrow0,
\end{equation}
is exact and every $C^p(\Sigma;\mathcal R)$ is flat over $R$. Therefore, the fan determines a flat resolution of minimal length.
\end{restatable}

 The construction is independent of the fan up to quasi-isomorphisms. When the fan is simplicial, the cellular flat resolution carries a cup product. When it is also centrally symmetric, the~antipodal map on $S^{d-1}$ induces an action compatible with this product. 

In dimension one the resolution takes the familiar form, used in an essential way in our previous work \cite{ruba2025witt}:
\begin{equation}
0\longrightarrow k[x^{\pm1}]
\longrightarrow k((x))\oplus k((x^{-1}))
\longrightarrow k[[x^{\pm1}]]
\longrightarrow0.
\end{equation}
Here the two middle terms correspond to the two points of $S^0$, or equivalently to the two asymptotic directions $+\infty$ and $-\infty$.

We next explain the Poincar\'e duality. A \emph{duality pair} $(M,N,\Omega)$ consists of finitely generated $R$-modules $M,N$ and a nondegenerate involution-sesquilinear form
\begin{equation}
\Omega\colon M\times N\longrightarrow R.
\end{equation}
The form determines embeddings $N\to M^*$ and $M\to N^*$. We modify the usual Ext and Tor groups only in degree zero by setting
\begin{subequations}
\begin{align}
\widetilde\Ext^0(\overline M, R)
&=\operatorname{coker}(N\longrightarrow \Hom_R(\overline M,R)),\\
\widetilde\Tor_0(\widehat R,M)
&=\ker(\widehat R \otimes M\longrightarrow \Hom_k(N,k)),
\end{align}
\end{subequations}
declaring $\widetilde \Ext^p$ and $\widetilde \Tor_p$ to agree with $\Ext^p_R(\overline M,R)$ and $\Tor^R_p(\widehat R,M)$ for $p>0$. Equivalently, the modified extension and torsion modules are the (co)homology modules of the mapping cones of
\begin{equation}
    N\longrightarrow \operatorname{RHom}_R(\overline M,R),
\qquad
\widehat R\otimes_R^{\mathbb L}M\longrightarrow \Hom_k(N,k).
\end{equation}

The two constructions are dual to each other:
\begin{equation}
    \widetilde \Tor_p(\widehat R,M) \cong \Hom_k(\widetilde \Ext^p(\overline M,R), k).
\end{equation}
The modified extension and torsion functors vanish for a unimodular pairing between projective modules. We will refer to them as \emph{duality defects.}

Using the Alexander-Whitney formula for cup products, we define canonical pairings between the modified Tor groups. Integration over the fundamental class $[S^{d-1}]$ provides an algebraic analogue of the usual Poincar\'e pairing.

\begin{restatable}{maintheorem}{Poincare}
    \label{thm:intro-poincare}
Let $(M,N,\Omega)$ be a duality pair. The cup product in the cellular complex \eqref{eq:augmented_exact_complex} induces sesquilinear pairings, for every $0\leq p\leq d-1$, 
\begin{equation}
\operatorname{Br}_p\colon
\widetilde\Tor_p(\widehat R,M)
\times
\widetilde\Tor_{d-p-1}(\widehat R,N)
\longrightarrow\widehat R.
\label{eq:intro-Br}
\end{equation}
We also have canonical homomorphisms
\begin{equation}
\kappa_p\colon
\widetilde\Tor_p(\widehat R,N)
\longrightarrow
\widetilde\Ext^{d-p-1}(\overline M,R).
\end{equation}
If every $\widetilde\Ext^q(\overline M,R)$ has finite length, then the maps $\kappa_p$ are isomorphisms and the pairings $\operatorname{Br}_p$ are perfect. In particular, we have induced isomorphisms
\begin{subequations}
    \begin{align}
        \widetilde \Tor_p(\widehat R, M) & \cong \Hom_k(\widetilde \Tor_{d-p-1}(\widehat R,N),k), \\
        \widetilde \Tor_{d-p-1}(\widehat R, N) & \cong \Hom_k(\widetilde \Tor_{p}(\widehat R,M),k).
    \end{align}
\end{subequations}
\end{restatable}

From a mathematical standpoint, the finite length hypothesis in Theorem \ref{thm:intro-poincare} describes the minimal departure from the setting of unimodular pairings between projective modules. Physically, it often characterizes models with mobile excitations. See Section~\ref{sec:physics}.

There are two natural situations in which the two entries of the pairing in \eqref{eq:intro-Br} are related to each other. First, let $P$ be a module carrying a nondegenerate $\varepsilon$-hermitian form, $\varepsilon \in \{ \pm 1 \}$. Then we have a duality pair $(P,P,\Omega)$, and~$\operatorname{Br}$ is a pairing between modules $\widetilde \Tor(\widehat R,P)$ in complementary degrees. If $d$ is odd, we obtain a middle-degree self-pairing on
\begin{equation}
    \widetilde\Tor_{\frac{d-1}{2}}(\widehat R,P).
\end{equation}

Secondly, suppose that $P$ is free with a unimodular $\varepsilon$-hermitian form, inducing an isomorphism $P \to \Hom(\overline P,R)$. Let $L\subset P$ be a Lagrangian submodule (not necessarily a direct summand). Setting $Q=P/L$, we obtain a duality between $Q$ and $L$. 
Composing $\operatorname{Br}$ with connecting homomorphisms $\widetilde \Tor_{p}(\widehat R,L) \cong \widetilde \Tor_{p+1}(\widehat R,Q)$, we obtain pairings $\operatorname{Br}^L$ between modules $\widetilde \Tor_\bullet(\widehat R,Q)$. For even~$d$, there is a self-pairing on
\begin{equation}
\widetilde\Tor_{\frac d2}(\widehat R,Q).
\end{equation}
The next theorem concerns the bilinear forms $b$ obtained by taking the constant terms of the self-pairings described above. Our main discovery is a natural construction of quadratic refinements. When $2$ is not invertible in $k$, one cannot simply construct a quadratic refinement of a~bilinear form by dividing its diagonal by two. This difficulty is overcome by replacing integration over the sphere with integration over the projective space. A technical obstruction arises because the relevant cochains on the sphere are not symmetric or skew-symmetric on the nose under the antipodal map, and thus do not directly descend to the quotient space. Using Steenrod's higher cup products, we find the homotopies correcting the failure of symmetry. Pairing the resulting equivariant cocycle with the twisted fundamental class of $\mathbb{RP}^{d-1}$ defines our quadratic functions. Passing from the sphere at infinity to projective space provides a~geometric substitute for division by two.

\begin{restatable}{maintheorem}{Refinement}
    \label{thm:intro-quadratic} 
The following statements hold.
\begin{enumerate}[(a)]
\item Let $P$ carry a nondegenerate $\varepsilon$-hermitian form and suppose that $d$ is odd. Let
\begin{equation}
A=\widetilde\Tor_{\frac{d-1}{2}}(\widehat R,P),
\end{equation}
and let $b\colon A\times A\to k$ be the scalar part of $\operatorname{Br}_{\frac{d-1}{2}}$. 

If
\(
\varepsilon(-1)^{\frac{d-1}{2}}=1,
\)
then $b$ is alternating. 

If
\(
\varepsilon(-1)^{\frac{d-1}{2}}=-1,
\)
then $b$ is symmetric and admits a distinguished quadratic refinement
\begin{equation}
q\colon A\longrightarrow k,
\end{equation}
explicitly defined in Section \ref{sec:quadratic}.
For $d=1$, the construction of $q$ additionally requires that the scalar part of $\Omega$ is alternating.

\item Let $P$ be free with a unimodular $\varepsilon$-hermitian form, and let $L\subset P$ be Lagrangian. Set $Q=P/L$, and suppose that $d$ is even. Let
\begin{equation}
A=\widetilde\Tor_{\frac d2}(\widehat R,Q),
\end{equation}
and let $b^L\colon A\times A\to k$ be the scalar part of $\operatorname{Br}^L_{\frac{d}{2}}$. 

If
\(
\varepsilon(-1)^{\frac d2}=-1,
\)
then $b^L$ is alternating. 

If
\(
\varepsilon(-1)^{\frac d2}=1,
\)
then $b^L$ is symmetric and admits a distinguished quadratic refinement
\begin{equation}
q^L\colon A\longrightarrow k,
\end{equation}
explicitly constructed in Section \ref{sec:quadratic}.
For $d=2$, the construction of $q^L$ additionally requires that the scalar part of the original form on $P$ is alternating.
\end{enumerate}
In both cases the quadratic refinement satisfies
\begin{equation}
q(mu)=m^2q(u),
\qquad
q(u+v)-q(u)-q(v)=b(u,v),
\end{equation}
with the analogous formula for $q^L$ and $b^L$.
\end{restatable}

The article is organized as follows. Section~\ref{sec:physics} elaborates on the physical motivations behind the theory. Section~2 introduces the cellular flat resolution of $\widehat R$, including the cup product and antipodal symmetry. Section~3 introduces defects of duality pairs and pairings announced in Theorem \ref{thm:intro-poincare}. The proof of Theorem \ref{thm:intro-poincare} given there relies on the analysis of a double complex, which is defined as the tensor product of a two-sided resolution associated with $\Omega$ and the cellular resolution of $\widehat R$.  
We~then specialize the result to modules with a self-pairing and to Lagrangian submodules. Section~4 constructs the quadratic refinements. The main ingredients there are equivariant (co)homology for the antipodal action on the sphere and~higher cup products. Appendix~A records facts about double complexes used throughout the article, while~Appendix~B collects facts about supports and dimensions of character modules.

\subsection{Motivation from physics} \label{sec:physics}

The algebraic structures considered in this paper arose from the study of translation-invariant Pauli stabilizer codes~\cite{CRSS,Gottesman}. These are quantum lattice models that have been studied in the context of quantum error correction and as tractable realizations of nontrivial topological order. 

Following the Laurent-polynomial formalism of Haah~\cite{haah2013commuting}, translation symmetry allows one to organize local Pauli operators into an $R$-module $P$.
The commutation relations of Pauli operators are encoded by a hyperbolic skew-hermitian form
\begin{equation}
\Omega\colon P\times P\longrightarrow R.
\end{equation}
The scalar part of $\Omega(p,q)$ encodes the commutator phase of operators corresponding to $p$ and $q$, while other coefficients record commutators with translates. A translation-invariant stabilizer group determines an $R$-submodule $L\subset P$. Since the stabilizers commute with one another, $L$~is isotropic. For the class of stabilizer codes considered here one further has
\begin{equation}
L=L^\perp.
\end{equation}
This is the exactness condition of~\cite{haah2013commuting}; in our terminology $L$ is a Lagrangian submodule~\cite{ruba2024homological,ruba2025witt}. This gives one of the two main examples of a duality pair considered in the paper. 

In two dimensions, a classification of Lagrangian stabilizer codes with prime-dimensional qudits was obtained in~\cite{haah2021classification}. Allowing composite qudit dimension, Ellison et al.~\cite{ellison2022pauli} constructed stabilizer code representatives for all Witt-trivial abelian anyon theories. Conversely, results in \cite{ruba2024homological,ruba2025witt} establish that every two-dimensional Lagrangian stabilizer code is characterized by a Witt-trivial abelian anyon theory. Physically, Witt-triviality is often interpreted as admitting a gapped boundary condition, see e.g.\ \cite{kapustin2011topological}. In three dimensions, the landscape of stabilizer codes is much richer and less understood. In particular, the relevant duality defects need not have finite length. Such models admit local excitations, called \emph{fractons}, that cannot be transported in arbitrary directions by the action of local operators. 
 
The second main example arises in the study of stabilizer codes in geometries with boundary. In~\cite{ruba2025witt}, we associated to a translation-invariant stabilizer state restricted to a half-space a boundary operator module $P_\partial$; related boundary constructions appear in~\cite{liang2024operator,schuster2023holographic, haah2023nontrivial}. Commutation of boundary operators gives a nondegenerate skew-hermitian form
\begin{equation}
\Omega_\partial\colon P_\partial\times P_\partial\longrightarrow R_\partial,
\end{equation}
where $R_\partial$ is the Laurent polynomial ring associated with translations parallel to the boundary. The form need not be unimodular, and in the terminology of \cite{ruba2024homological} $P_\partial$  
is called quasi-symplectic. Therefore, one obtains the duality pair $(P_\partial,P_\partial,\Omega_\partial)$.
    
Another source of duality pairs is the algebraic theory of planon-only fracton orders developed in~\cite{wickenden2024planon,wickenden2025excitation}. There the superselection sectors form a finitely generated module $S$ over a~one-variable Laurent polynomial ring. For a theory of finite fusion order $n$, the mutual statistics of planons can be assembled into an $R$-valued hermitian form
\begin{equation}
B\colon S\times S\longrightarrow \mathbb Z_n[t^{\pm1}].
\end{equation}
The condition of $p$-modularity, or equivalently remote detectability, is the nondegeneracy of this form. Thus $(S,S,B)$ is a duality pair. 

For the Lagrangian stabilizer codes described above, we introduced
in~\cite{ruba2024homological} the \emph{charge modules}
\begin{equation}
Q^i=\Ext_R^{i+1}(\overline{P/L},R),
\end{equation}
which are instances of the duality defects considered in
the present paper. If the charge modules have finite length, they
describe mobile excitations: point-like for $i=0$, string-like for $i=1$, and so on. Moreover, the complementary modules $Q^i$ and $Q^{d-2-i}$ carry a braiding pairing.
In~\cite{ruba2024homological}, we constructed this pairing using an algebraic \v{C}ech complex, but left its nondegeneracy open.
Theorem~\ref{thm:intro-poincare} closes this gap: applied to the
duality pair $(P/L,L)$, it~shows that the braiding pairing is perfect
whenever the charge modules have finite length.

The analysis of charge modules has been taken up in recent work on the classification of CSS codes \cite{SongKoszul}. In \cite{yang2026classification}, invariants of stabilizer codes have been compared with framed topological quantum field theories. We mention also the study \cite{geiko2025} of invertible stabilizer codes, which are characterized by the vanishing of all charge modules. 

Motivated by spin--statistics and its higher-dimensional analogues, we also expected the middle-dimensional braiding to carry additional structure.  We proposed in~\cite{ruba2024homological} that in dimensions divisible by $4$ the middle-dimensional braiding should be alternating. In dimensions congruent to $2$ modulo $4$, we expected it to admit a distinguished quadratic refinement, generalizing the topological spin familiar in two dimensions~\cite{levin2003fermions,haah2021classification}. Theorem \ref{thm:intro-quadratic} confirms these expectations. 

In dimension two, an explicit formula for the quadratic form, often called the T-junction, was already known~\cite{levin2003fermions,haah2021classification}. Our work extends this topological spin function to higher dimensions and provides its natural geometric interpretation.

Other works on generalized statistics of extended or fractonic excitations in stabilizer codes include~\cite{FractonStatistics,kobayashi2024generalized,feng2026pauli}.

The Lagrangian case and the self-pairing case are related by the bulk--boundary correspondence arising from the half-space construction, see~\cite{ruba2025witt} for a precise definition. We~expect a direct relation between defects of $(P/L,L)$ over $R_\partial[x_d^{\pm 1}]$ and those of $(P_\partial,P_\partial)$ over $R_\partial$. In~two dimensions, this prediction has been verified in~\cite{liang2024operator,ruba2025witt}.

\section{Resolution of the character module}
\subsection{Laurent rings and character modules}

Fix an integer $n\geq 2$ and let $k=\mathbb Z/n\mathbb Z$. We set $\Lambda=\mathbb Z^d$ and $R=k[\Lambda]=k[x_1^{\pm1},\ldots,x_d^{\pm1}]$. For $\lambda=(\lambda_1,\ldots,\lambda_d)\in \Lambda$, we write $x^\lambda=x_1^{\lambda_1}\cdots x_d^{\lambda_d}$. The ring $R$ is equipped with the \emph{involution} $\overline{x^\lambda}=x^{-\lambda}$. If $r=\sum_{\lambda\in \Lambda}r_\lambda x^\lambda$, we call $r_0$ the \emph{scalar part} of $r$.

We write $\widehat R$ for the $R$-module of unrestricted formal series
\begin{equation}
\widehat R
=
\left\{
\sum_{\lambda \in \Lambda}f_\lambda x^\lambda \ \middle|\ f_\lambda \in k
\right\}.
\end{equation}
Multiplication by Laurent polynomials gives $\widehat R$ the structure of an $R$-module. It is not a ring, since the product of two unrestricted series is not defined in general. We emphasize that $\widehat R$ \emph{does not} denote an adic completion. For $f=\sum_\lambda f_\lambda x^\lambda\in\widehat R$, we set 
\begin{equation}
 \operatorname{supp}(f)=\{\lambda \in \Lambda\mid f_\lambda \neq0\}.   
\end{equation}

Let $R^\#=\operatorname{Hom}_k(R,k)$, with $R$-module structure $(r\varphi)(s)=\varphi(\overline r s)$. We have the isomorphism
\begin{equation}
\widehat R\xrightarrow{\ \sim\ }R^\#,
\qquad
f\longmapsto\bigl(r\longmapsto(\overline r f)_0\bigr).
\end{equation}
Equivalently, the functional corresponding to $f=\sum_\lambda f_\lambda x^\lambda$ sends $x^\lambda$ to $f_\lambda$. 

Recall that the character module of $R$ is defined as $\operatorname{Hom}_{\mathbb Z}(R,\mathbb Q/\mathbb Z)$, regarded as an $R$-module by $(r\varphi)(s)=\varphi(\overline r s)$. Since $R$ is annihilated by $n$, every such character takes values in the subgroup $(\mathbb Q/\mathbb Z)[n] := \{ a \in \mathbb Q/\mathbb Z \, | \, na =0 \pmod{\mathbb Z} \}$. The isomorphism $k\to(\mathbb Q/\mathbb Z)[n]$ given by $a\mapsto a/n\pmod{\mathbb Z}$ therefore induces
\begin{equation}
\operatorname{Hom}_{\mathbb Z}(R,\mathbb Q/\mathbb Z)
\cong
\operatorname{Hom}_k(R,k)
=
R^\#.
\end{equation}
Together with the preceding isomorphism $\widehat R\cong\operatorname{Hom}_k(R,k)$, this identifies both $\widehat R$ and $R^\#$ with the character module of $R$. We will use the two notations interchangeably.

For an $R$-module $M$, let $\overline M$ denote the $R$-module with the same underlying abelian group as $M$ and scalar multiplication $r\overline m=\overline{\overline r m}$. Thus, homomorphisms $\overline M\to L$ may be identified with antilinear maps $M\to L$. We define
\begin{equation}
M^*=\operatorname{Hom}_R(\overline M,R),
\qquad
\widehat M=\widehat R\otimes_RM,
\qquad
M^\#=\operatorname{Hom}_R(\overline M,\widehat R).
\end{equation}

\begin{defn} \label{def:sesquilinear}
Let $M$ and $N$ be $R$-modules. A sesquilinear pairing of $M$ and $N$ is an additive map $\Omega:M\times N\to R$ satisfying
$\Omega(rm,sn)=\overline r s\,\Omega(m,n)$ for $r,s\in R$, $m\in M$, and $n\in N$.

The pairing is nondegenerate if the maps
\begin{equation}
\begin{aligned}
\lambda_N &:N\longrightarrow M^*,
&n&\longmapsto\Omega(\mathord\cdot,n),\\
\lambda_M &:M\longrightarrow N^*,
&m&\longmapsto\overline{\Omega(m,\mathord\cdot)}
\end{aligned}
\end{equation}
are injective. It is called unimodular if both maps are isomorphisms.

A \emph{duality pair} is a triple $(M,N,\Omega)$, where $M,N$, are finitely generated $R$-modules and $\Omega$ is a~nondegenerate sesquilinear form.
\end{defn}

If $M=N$ and $\varepsilon\in\{\pm1\}$, we say $\Omega$ is $\varepsilon$-hermitian if $\Omega(m,m')=\varepsilon \, \overline{\Omega(m',m)}$.

The inclusion $R\subset\widehat R$ induces an inclusion $M^*\subset M^\#$. Composition with the scalar-part map $\widehat R\to k$ gives an isomorphism $M^\#\cong\operatorname{Hom}_k(M,k)$. Observe that every homomorphism of abelian groups $M\to k$ is automatically $k$-linear. Therefore,
\begin{equation}
\operatorname{Hom}_{\mathrm{Ab}}(M,k)
\cong
\operatorname{Hom}_R(\overline M,R^\#).
\end{equation}
We refer to $M^\#$ as the character module of $M$. It also coincides with the Pontryagin dual of the underlying abelian group, using the identification of $k$ with $(\mathbb Q/\mathbb Z)[n]$. 

The preceding isomorphism has a simple description. Given an additive map $\varphi:M\to k$, the~corresponding $R$-linear map $\widetilde\varphi:\overline M\to\widehat R$ is
\begin{equation}
\widetilde\varphi(\overline m)
=
\sum_{\lambda\in \Lambda}\varphi(x^\lambda m)x^\lambda.
\end{equation}
The inverse correspondence takes the scalar part: $\varphi(m)=\widetilde\varphi(\overline m)_0$.

Suppose that $M$ is finitely generated, and choose generators $m_1,\ldots,m_t$. We define the translation support of $\varphi$ by
\begin{equation}
\operatorname{supp}_\Lambda(\varphi)
=
\{\lambda \in \Lambda\mid \varphi(x^\lambda m_j)\neq0
\text{ for some }j\}.
\end{equation}
We say that $\varphi$ has bounded support if this set is finite. This condition is independent of the choice of a finite generating set. 

\begin{remark}
Under the isomorphism
$\operatorname{Hom}_{\mathrm{Ab}}(M,k)\cong
\operatorname{Hom}_R(\overline M,R^\#)$, the module
$M^* \subset M^{\#}$ consists precisely of the characters with bounded support. Indeed, $\widetilde\varphi$ takes values in $R$ if and only if the series $\widetilde\varphi(\overline m_j)$ has finite support for a finite collection of generators $m_j$.
\end{remark}

We also note that an $R$-valued sesquilinear form is determined by its scalar part. Indeed,
\begin{equation}
\Omega(m,n)
=
\sum_{\lambda \in \Lambda}\Omega(x^\lambda m,n)_0 x^\lambda.
\end{equation}
For every $n\in N$, the character $\Omega(\mathord\cdot,n)_0:M\to k$ has bounded support because the corresponding map $\overline M\to R^\#$ takes values in $R$. Conversely, a $k$-valued pairing with this bounded-support property determines an $R$-valued sesquilinear form by the formula above.

The ring $k$ is self-injective, so the functor $(-)^\#$ is exact. In other words, $\widehat R$ is an injective $R$-module.

\begin{lem} \label{lem:star_hat_sharp}
If $F$ is a finitely generated free $R$-module, then evaluation induces a natural isomorphism
\begin{equation}
\widehat R\otimes_RF^*
\cong
F^\#.
\end{equation}
\end{lem}

\begin{proof}
It is enough to consider $F=R$, in which case the claim is the identification $\widehat R\cong R^\#$ above.
\end{proof}

The functor $(-)^*$ is left exact, while the functor $M\mapsto\widehat M$ is right exact. Their derived functors are $\operatorname{Ext}^p_R(\overline M,R)$ and $\operatorname{Tor}^R_p(\widehat R,M)$, respectively.

\begin{prop} \label{prop:Ext-Tor-duality}
For every finitely generated $R$-module $M$ and every $p\geq0$, there is a natural isomorphism
\begin{equation}
\operatorname{Tor}^R_p(\widehat R,M)
\cong
\operatorname{Ext}^p_R(\overline M,R)^\#.
\end{equation}
\end{prop}
\begin{proof}
Choose a resolution $F_\bullet\to M$ by finitely generated free modules. Degreewise evaluation identifies $\widehat R\otimes_RF_\bullet$ with the character dual of $\operatorname{Hom}_R(\overline F_\bullet,R)$. The claim follows from the exactness of $(-)^\#$.
\end{proof}
  
\subsection{Fans and the sphere at infinity}
Following the terminology of toric geometry~\cite{cox2024toric}, we~make the following definitions.
\begin{defn}
Recall that
    $\Lambda=\mathbb Z^d.$ Let
\begin{equation}
        \Lambda_{\mathbb R}=\Lambda\otimes_{\mathbb Z}\mathbb R\cong \mathbb R^d .
\end{equation}

A rational polyhedral cone in \(\Lambda_{\mathbb R}\), or simply a cone, is a subset of $\Lambda_{\mathbb R}$ of the form
\begin{equation}
        \sigma
    =
    \mathbb R_{\geq 0}v_1+\cdots+\mathbb R_{\geq 0}v_k
\end{equation}

for some lattice vectors \(v_1,\dots,v_k\in \Lambda\). We adopt the convention that $\{ 0 \}$ is a cone, whereas $\emptyset$ is not. The dimension of $\sigma$ is defined as the dimension of the linear space spanned by $\sigma$. 

A cone \(\sigma\) is called
strongly convex if
$
    \sigma\cap(-\sigma)=\{0\},
$
or equivalently, if \(\sigma\) contains no nonzero linear subspace.

A subset \(\tau \subset \sigma\) is called a face of \(\sigma\) (denoted \(\tau \preceq \sigma\)) if there exists a linear functional \(u  \) on $\Lambda_{\mathbb R}$ such that \(u(v) \geq 0\) for all \(v \in \sigma\) and
\begin{equation}
       \tau = \{ v \in \sigma \mid u(v) = 0 \}. 
\end{equation}

\end{defn}
\begin{defn}
A fan \(\Sigma\) in \(\Lambda_{\mathbb R}\) is a finite collection of strongly convex
rational polyhedral cones satisfying the following two conditions:
\begin{enumerate}
    \item If \(\sigma\in\Sigma\) and \(\tau\preceq\sigma\) is a face of
    \(\sigma\), then \(\tau\in\Sigma\).
    \item If \(\sigma,\tau\in\Sigma\), then \(\sigma\cap\tau\) is a face of
    both \(\sigma\) and \(\tau\).
\end{enumerate}

A fan \(\Sigma\) is called complete if it covers $\Lambda_{\mathbb R}$, that is:
\begin{equation}  \bigcup_{\sigma\in\Sigma}\sigma=\Lambda_{\mathbb R}.
\end{equation}

\end{defn}

\begin{defn}
A cone $\sigma$ is called simplicial if it is generated by a linearly independent set. A~fan $\Sigma$ is called simplicial if every cone $\sigma \in \Sigma$ is simplicial.
\end{defn}
\begin{defn} \label{def:central_sym}
We call a fan $\Sigma$ centrally symmetric if it is invariant under the antipodal map, i.e.\ if $\sigma\in\Sigma$ implies $-\sigma\in\Sigma$.
\end{defn}

We use fans somewhat differently from
their standard application in toric geometry. We will not use fans to encode toric 
varieties, and their precise polyhedral geometry will play no essential role. Rather, our goal is to obtain cellular structures of the \textit{sphere at infinity}.
\begin{defn}
We define the sphere at infinity of \(\Lambda_{\mathbb R}\) to be
\begin{equation}
         S_\infty(\Lambda_{\mathbb R})
    :=
    (\Lambda_{\mathbb R}\setminus\{0\})/\mathbb R_{>0},
\end{equation}

where \(\mathbb R_{>0}\) acts by scalar multiplication. Therefore,
\(S_\infty(\Lambda_{\mathbb R})  \) parametrizes oriented rays, or equivalently
asymptotic directions, in \(\Lambda_{\mathbb R}\). We have 
\begin{equation}
   S_\infty(\Lambda_{\mathbb R})\simeq  S^{d-1}, 
\end{equation}
and for the rest of the paper we denote the sphere at infinity by $S^{d-1}$. 
\end{defn}

Let $\Sigma$ be a complete fan in \(\Lambda_{\mathbb R}\). We call $\{0\}$ the trivial cone, and set
\begin{equation}
    \Sigma^\times := \Sigma \setminus \big\{ \{0\} \big \}  =\bigcup_{p =0}^{d-1} \Sigma_p,
\end{equation} 
where $\Sigma_p$ is the collection of all $(p+1)$-dimensional cones in $\Sigma$. The degree shift reflects the fact that every cone $\sigma \in \Sigma_p$ determines a $p$-cell in the sphere at
infinity, and $\Sigma^\times$
is the set of all cells of a polyhedral cellular decomposition of $S^{d-1}$. 

\begin{defn}
We introduce the groups of cellular $p$-chains of $S^{d-1}$:
\begin{equation}
C_p(\Sigma;\mathbb Z)
:=
\bigoplus_{\sigma\in \Sigma_p} \mathbb Z \cdot [\sigma].
\end{equation}
Thus, $C_p(\Sigma;\mathbb Z)$ is the free $\mathbb Z$-module with basis $\{ [\sigma] \}_{\sigma \in \Sigma_p}$. 

The cellular boundary maps define the structure of a chain complex:
\begin{equation}
C_\bullet(\Sigma;\mathbb Z)
:
C_{d-1}(\Sigma;\mathbb Z)
\longrightarrow \cdots \longrightarrow
C_0(\Sigma;\mathbb Z).
\end{equation}
Dually, we introduce cellular $p$-cochains:
\begin{equation}
C^p(\Sigma;\mathbb Z)
:=
\Hom\bigl(C_p(\Sigma;\mathbb Z),\mathbb Z\bigr),
\end{equation}
with coboundary maps defining a cochain complex
\begin{equation}
C^\bullet(\Sigma;\mathbb Z)
:
C^0(\Sigma;\mathbb Z)
\longrightarrow \cdots \longrightarrow
C^{d-1}(\Sigma;\mathbb Z).
\end{equation}
\end{defn}

\begin{defn} \label{def:fund}
    Let us fix an orientation of $S^{d-1}$ and of every $\sigma \in \Sigma_{d-1}$. We compare the chosen orientation of $\sigma$ with the induced orientation, and set
    \begin{equation}
        [S^{d-1}:\sigma] = \begin{cases} +1, & \text{if the two orientations agree}, \\
        -1, & \text{if the two orientations disagree}.
        \end{cases}
    \end{equation}
    The fundamental cycle is the closed $(d-1)$-chain
    \begin{equation}
        [S^{d-1}] = \sum_{\sigma \in \Sigma_{d-1}}  [S^{d-1}:\sigma]  [\sigma] \in C_{d-1}(\Sigma; \mathbb Z). 
        \label{eq:sphere_fund_class}
    \end{equation}
    Its homology class generates $H_{d-1}(\Sigma;\mathbb Z) \cong \mathbb Z$.

    In the special case $d=1$, $S^{d-1}$ consists of two points, and we define $[S^{d-1}]$ as their formal difference (cf.\ Example \ref{exmp:1d_resol} below).
\end{defn}

\subsection{Cone-supported algebras and the cellular flat resolution}

\begin{defn} \label{def:cal_R}
We define a map of posets $\mathcal R$ from the poset of rational polyhedral cones in~$\Lambda_{\mathbb R}$, ordered by inclusion, to the poset of $R$-modules, also ordered by inclusion,
\begin{equation}
        \mathcal R :
    \{\text{rational polyhedral cones in } \Lambda_{\mathbb R}\}
    \longrightarrow
    \{ R\text{-modules} \},
\end{equation}

by setting
\begin{equation}
        \mathcal R(\sigma)
    :=
    \left\langle
        f\in \widehat R :
        \operatorname{supp}(f)\subseteq \sigma
    \right\rangle_R.
\end{equation}

Here $\langle S\rangle_R$ denotes the $R$-submodule of $\widehat R$ generated by $S$.
For every inclusion $\tau \subset \sigma$, there is a~natural inclusion $\mathcal R(\tau) \hookrightarrow \mathcal R(\sigma)$.
\end{defn}
Let us give a few examples:
\begin{itemize}
    \item $\mathcal R(\Lambda_{\mathbb R})=\widehat R$,
    \item $\mathcal R(\{ 0 \})=R$,
    \item $\mathcal R(\mathbb R_{ \geq 0}) = k((t))$ for $\Lambda = \mathbb Z$.
\end{itemize}

\begin{prop}\label{prop:flatness}
    The functor $\mathcal R$ restricts to 
    \begin{equation}
            \{\text{strongly convex rational polyhedral cones} \}
    \longrightarrow
    \{\text{flat } R\text{-algebras}\}.
    \end{equation}
 In other words, $ \mathcal R(\sigma)$ is a flat $R$-algebra if $\sigma$ is strongly convex. 
\end{prop}

We state without proof a well-known fact needed to prove Proposition \ref{prop:flatness}.

\begin{lem}\label{lem:flatness_base_change}
Let $A \to B$ be a homomorphism of commutative rings, and let $M$ be a flat
$A$-module. Then
$B\otimes_A M$
is a flat $B$-module.
\end{lem}

\begin{proof}[Proof of Proposition \ref{prop:flatness}]
If $f = \sum_\lambda f_\lambda x^\lambda$ and $g = \sum_\lambda g_\lambda x^\lambda$, we put
\begin{equation}
    fg = \sum_\lambda \left( \sum_\mu f_{\lambda-\mu} g_\mu \right) x^\lambda.
\end{equation}
The coefficient of $x^\lambda$ in this expression is a sum with finitely many nonzero terms because $\sigma$ is strongly convex. With this multiplication, $\mathcal R(\sigma)$ is a~commutative ring containing $R$ as a~subring. 

Let us show that $\mathcal R(\sigma)$ is a flat $R$-module. Let $S = \sigma \cap \Lambda$. Then $-S \cap S = \{ 0 \}$ by the strong convexity of $\sigma$. Let $R_0=k[S]$ be the semigroup ring of $S$:
\begin{align}
    R_0 = \{ f \in  R \, | \, \operatorname{supp} f \subset S \}
     =\left\{\sum_{\lambda \in S} a_\lambda x^\lambda \middle|a_\lambda=0\text{ for all but finitely many }\lambda \right\}. 
\end{align}
By Gordan's lemma (see e.g.\ \cite[Lemma~2.9]{BrunsGubeladze2009}), $S$ is a finitely generated subsemigroup of~$\Lambda$, and therefore $R_0$ is a Noetherian ring. Let $I \subset R_0$ be the ideal generated by all nonconstant monomials $x^\lambda$ with $\lambda \in S$, and let $R_1$ be the $I$-adic completion of $R_0$. Then 
\begin{equation}
    R_1 = \{ f \in \widehat R \, | \, \operatorname{supp} f \subset S \} = \left \{ \sum_{\lambda \in S} a_\lambda x^\lambda \right\}
\end{equation}
is flat over $R_0$ \cite[Proposition 10.14]{AtiyahMacdonald}. We~observe that $\mathcal R(\sigma) \cong R_1 \otimes_{R_0} R$ and invoke Lemma \ref{lem:flatness_base_change}. 
\end{proof}

Let $\Sigma$ be a complete fan. Let us choose an orientation of each cone $\sigma\in\Sigma^\times$, or equivalently orientations of all cells of $S^{d-1}$ (we will further elaborate on our preferred choices of orientations in Subsection \ref{sec:cup} below). Specifically, an ordered basis $(v_1, \dots, v_{k})$ of a~tangent space of a cell is positively oriented if and only if $(\nu, v_1, \dots, v_{k})$ is positively oriented for the corresponding cone, where $\nu$ is an outward pointing radial vector.

For $\tau\preceq\sigma$ with $\dim \tau=\dim\sigma-1$, let
\(
[\sigma:\tau]\in  \{\pm 1 \}
\)
denote the cellular incidence number. More explicitly, the orientation of $\sigma$ induces an orientation of its boundary, and thus of every codimension-one face $\tau$. We compare the induced orientation with the chosen orientation of $\tau$, and set
\begin{equation}
[\sigma:\tau]=
\begin{cases}
+1, & \text{if the two orientations agree,}\\
-1, & \text{if the two orientations disagree.}
\end{cases}
\end{equation}

\begin{defn} \label{def:cochains_with_coeff}
We define the $R$-module of cellular $p$-cochains with coefficients in
$\mathcal R$ by
\begin{equation}
C^p(\Sigma;\mathcal R)
:=
\prod_{\sigma\in\Sigma_p}\mathcal R(\sigma).
\end{equation}
The differential
\begin{equation}
\delta^p:C^p(\Sigma;\mathcal R)\longrightarrow C^{p+1}(\Sigma;\mathcal R)
\end{equation}
is defined as follows: $\delta^p c =  (\delta^p c)_{\sigma \in \Sigma_{p+1}}$ has the $\sigma$-component 
\begin{equation}
(\delta^p c)_\sigma
:=
\sum_{\substack{\tau\preceq\sigma, \\ \tau \in \Sigma_p}}
[\sigma:\tau] \, \iota_{\tau,\sigma}(c_\tau),
\end{equation}
where
\(
\iota_{\tau,\sigma}:\mathcal R(\tau)\hookrightarrow\mathcal R(\sigma)
\)
is the natural inclusion associated to the face relation
$\tau\preceq\sigma$. In this way we obtain the cellular cochain complex
\begin{equation}
0
\longrightarrow
C^0(\Sigma;\mathcal R)
\xrightarrow{\delta^0}
C^1(\Sigma;\mathcal R)
\xrightarrow{\delta^1}
\cdots
\xrightarrow{\delta^{d-2}}
C^{d-1}(\Sigma;\mathcal R)
\longrightarrow
0 .
\end{equation}
\end{defn}

We will usually omit the inclusion maps $\iota$ from the notation. 

There is a natural pairing between cellular chains and cellular cochains:
\begin{equation}
\langle -,-\rangle \colon
C_{p}(\Sigma;\mathbb Z)\times C^p(\Sigma;\mathcal R)
\longrightarrow \widehat R .
\end{equation}
Explicitly, if
\begin{equation}
\gamma=\sum_{\sigma\in\Sigma_p} n_\sigma[\sigma]
\in C_{p}(\Sigma;\mathbb Z),
\qquad
c=(c_\sigma)_{\sigma\in\Sigma_p}
\in C^p(\Sigma;\mathcal R),
\end{equation}
then
\begin{equation}
\langle \gamma,c\rangle
:=
\sum_{\sigma\in\Sigma_p} n_\sigma c_\sigma
\in \widehat R,
\end{equation}
where each $c_\sigma\in\mathcal R(\sigma)$ is regarded as an element of $\widehat R$ via the natural inclusion $\mathcal R(\sigma)\hookrightarrow \widehat R$. We will denote pairing with $[S^{d-1}]$ as integration:
\begin{equation}
    \int c = \langle [S^{d-1}] ,c\rangle = \sum_{\sigma \in \Sigma_{d-1}} [S^{d-1}:\sigma] c_\sigma. 
    \label{eq:integral}
\end{equation}

\begin{defn} \label{def:augmentation}
The complex $C^\bullet(\Sigma;\mathcal R)$
is doubly augmented by the maps
\begin{equation}
R
\xrightarrow{\Delta}
C^0(\Sigma;\mathcal R), \qquad \text{and} \qquad C^{d-1}(\Sigma;\mathcal R)
\xrightarrow{\int}
\widehat R .  
\end{equation}
The first map is the diagonal inclusion
\begin{equation}
\Delta(r)=( r)_{\rho\in\Sigma_0},
\end{equation}
where $R$ is included in each $\mathcal R(\rho)$, and the rays $\rho \in \Sigma_0$ are oriented outward (cf.\ Definition \ref{def:cochains_with_coeff}). The second map is the integration defined in \eqref{eq:integral}.
\end{defn}

\begin{exmp} \label{exmp:1d_resol}
    In dimension $d=1$, let $\Lambda=\mathbb Z$ and $\Sigma=\{\{0\}, \sigma_+, \sigma_- \} $, where $\sigma_{\pm} = \pm \mathbb R_{\geq 0}$. Then $\mathcal R(\sigma_\pm)= k((x^\pm))$. Choosing outward orientations for the cones $\sigma_\pm$, the augmented complex is 
    \begin{equation} 
        0\longrightarrow k[x^\pm]\xrightarrow{\Delta} k((x))\oplus k((x^{-1})) \xrightarrow{\int} k[[x^\pm]]\longrightarrow 0, 
    \end{equation}
    with $\Delta(r)=(r, r)$ and $\int(c_+, c_-)=c_+ - c_-.$ 
    
    This complex is used in Section~3 of~\cite{ruba2025witt} to define quadratic forms associated to quasi-symplectic modules over $k[x^\pm]$ (i.e.\ for $d=1$).
\end{exmp}

\begin{exmp}
    For any dimension $d$, there exists a complete fan $\Sigma$ in $ \mathbb R^d$ whose top-dimensional cones are the coordinate orthants. For every $k \in \{ 0, \dots, d \}$, the fan $\Sigma$ contains $\binom{d}{k} 2^k$ cones of dimension~$k$. Cases $d=1$ and $d=2$ are discussed in Examples \ref{exmp:1d_resol} above and \ref{exmp:2d_resol} below.
\end{exmp}

\begin{remark} \label{rmk:orientations}
The choices of orientations for $1$-dimensional cones are made so that the augmentation map $\Delta$ includes no negative signs. For $d>1$, it~is also possible to orient every top-dimensional cone so that $\int$ is given by summation with no negative signs. These two conventions are incompatible for $d=1$ because the $0$-cells and top-dimensional cells coincide there (see Example \ref{exmp:1d_resol}). In Subsection~\ref{sec:cup}, we adopt an orientation convention induced by an ordering of $\Sigma_0$, which we use throughout the remainder of the article; this is possible if $\Sigma$ is simplicial.
\end{remark}
The following result states a key connection between the cellular homology of the sphere at infinity and commutative algebra over $R$.

\begin{proof}[Proof of Theorem \ref{thm:cellular}]
We will prove exactness of the complex
 \begin{align}
 0
 \longrightarrow
 R
 \xrightarrow{\Delta}
 C^0(\Sigma;\mathcal R)
  \xrightarrow{\delta^0}
 C^1(\Sigma;\mathcal R)
 \xrightarrow{\delta^1}
 \cdots 
 \xrightarrow{\delta^{d-2}}
 C^{d-1}(\Sigma;\mathcal R)
 \xrightarrow{\int}
 \widehat R
 \longrightarrow
 0 .
 \end{align}
by introducing a filtration of the complex and showing that each associated graded piece is exact.

We first recall two elementary facts about supports. For a cone
$\sigma\in\Sigma$, an~element $f\in\widehat R$ belongs to
$\mathcal R(\sigma)$ if and only if there exists a finite set $A\subset \Lambda$
such that
\(
    \operatorname{supp}(f)
    \subseteq
    \bigcup_{a\in A}(a+\sigma).
\)
Consequently, if $\alpha,\beta_1,\ldots,\beta_s\in\Sigma$, then
\begin{equation}
    \mathcal R(\alpha)\cap
    \left(\sum_{i=1}^s\mathcal R(\beta_i)\right)
    =
    \sum_{i=1}^s\mathcal R(\alpha\cap\beta_i).
\end{equation}
Indeed, the inclusion from right to left is immediate. For the reverse
inclusion, suppose that $f \in  \mathcal R(\alpha)\cap
    \left(\sum_{i=1}^s\mathcal R(\beta_i)\right)$. Then
$\operatorname{supp}(f)$ is contained in a finite union of sets of the
form
\(
    (a+\alpha)\cap(b+\beta_i).
\)
This set is a rational polyhedron with recession cone
$\alpha\cap\beta_i$. A standard consequence of Gordan's lemma \cite[Theorem~2.12]{BrunsGubeladze2009} states 
that lattice points of $(a+\alpha)\cap(b+\beta_i)$ are contained in a finite union of translates
of $\alpha\cap\beta_i$. Choosing a partition of the support of $f$ among
these finitely many sets, we~obtain a~decomposition
$f=\sum_{i=1}^s f_i$, with 
    $f_i\in\mathcal R(\alpha\cap\beta_i)$,
which proves the claim.

Now choose a total ordering of $\Sigma$,
\(
    \eta_0 < \eta_1 < \ldots < \eta_N,
\)
which refines the partial order given by inclusion: if $\eta_j \preceq \eta_k$, then $j \leq k$. In~particular, $\eta_0=\{0\}$. For $\sigma\in\Sigma$, set
\begin{equation}
    F_j\mathcal R(\sigma)
    :=
    \sum_{\substack{i\leq j\\ \eta_i\preceq\sigma}}
    \mathcal R(\eta_i),
    \qquad
    F_{-1}\mathcal R(\sigma):=0, \qquad F_jC^p(\Sigma;\mathcal R)
    :=
    \prod_{\sigma\in\Sigma_p}F_j\mathcal R(\sigma).
\end{equation}
Furthermore, for $j \geq 0$ we put
\begin{subequations}
\begin{align}
    F_j\widehat R
   & :=
    \sum_{i\leq j}\mathcal R(\eta_i),
   & 
    F_{-1}\widehat R &:=0, \\
    F_j R &:= R, & F_{-1} R& := 0 .
\end{align}
\end{subequations}
This defines a finite exhaustive filtration of the complex \eqref{eq:augmented_exact_complex}. The differentials, including the maps $\Delta$ and $\int$, preserve the filtration.

Next, we describe the $j$-th graded piece of the filtration. Let
\(
    \tau:=\eta_j,
    \)
and~define:
\begin{equation}
    Q_\tau
    :=
    \mathcal R(\tau)\Big/
    \sum_{\gamma\prec\tau}\mathcal R(\gamma),
\end{equation}
where the sum runs over the proper faces of $\tau$. In particular, $Q_\tau = R$ for $j=0$. Then:
\begin{equation}
    \frac{F_j\mathcal R(\sigma)}
    {F_{j-1}\mathcal R(\sigma)}
    \cong
    \begin{cases}
        Q_\tau, & \tau\preceq\sigma,\\
        0, & \tau\npreceq\sigma.
    \end{cases}
\end{equation}
We also have
\begin{equation}
    \frac{F_j\widehat R}{F_{j-1}\widehat R}
    \cong Q_\tau, \qquad  \frac{F_j R}{F_{j-1} R}=\begin{cases}
        Q_\tau, & j = 0, \\
        0, & j \neq 0.
    \end{cases}.
\end{equation}

We grade the complex \eqref{eq:augmented_exact_complex} so that $R$ lies in degree $-1$ and $\widehat R$ in degree $d$. Letting $r = \dim \tau$, the $j$-th graded piece of the
filtered complex is
\begin{equation}
    0 \longrightarrow Q_\tau
    \longrightarrow
    \prod_{\substack{\sigma\in\Sigma_r \\
    \tau\prec\sigma}}Q_\tau
    \longrightarrow \cdots 
    \longrightarrow
    \prod_{\substack{\sigma\in\Sigma_{d-1}\\
    \tau\prec\sigma}}Q_\tau
    \longrightarrow Q_\tau
    \longrightarrow 0,
    \label{eq:graded_piece}
\end{equation}
where the first copy of $Q_\tau$ occurs in degree $r-1$.
If $r=d$, the complex in \eqref{eq:graded_piece} reduces to
\begin{equation}
    0\longrightarrow Q_\tau
    \xrightarrow{\ \pm 1\ }
    Q_\tau\longrightarrow 0,
\end{equation}
which is clearly exact.

If $r<d$, the cones properly containing $\tau$ define a cellular structure of a~sphere of dimension $d-r-1.$ To see this, let $L_\tau $ be the linear span of $\tau$. We consider the quotient space $\Lambda_{\mathbb R}/ L_\tau \cong \mathbb R^{d-r}$. The collection of all cones $\sigma \in \Sigma$ which contain $\tau $ as a face, projected to $\Lambda_{\mathbb R}/L_\tau$ via the quotient map, defines a complete fan $\Sigma^{(\tau)}$ in $\Lambda_{\mathbb R} / L_\tau$, and therefore a corresponding cell structure of the sphere at infinity $S_\infty (\Lambda_{\mathbb R}/L_\tau) \cong S^{d-r-1}$.  

The associated graded piece \eqref{eq:graded_piece} coincides with the doubly augmented
cellular cochain complex of $S_\infty (\Lambda_{\mathbb R}/L_\tau)$ with constant coefficients
$Q_\tau$. It is exact because the augmentation kills the only two nontrivial cohomology groups of the sphere. 

We have constructed a flat resolution of length $d$. One can verify that no shorter flat resolution exists by observing that $\Tor_d(\widehat R,R/(x_1-1,\dots,x_d-1)) \neq 0$.
\end{proof}

Later, we will fix a complete fan $\Sigma$ and use the cellular flat resolution of $\widehat R$ repeatedly. Therefore, we introduce the following shorthand notation, using homological grading:
\begin{equation}
T_\bullet : T_d \longrightarrow T_{d-1} \longrightarrow \cdots \longrightarrow T_0,
\label{eq:T_complex}
\end{equation}
where $T_p := C^{d-p-1}(\Sigma; \mathcal R)$ and $T_d = R$.

We will now briefly address the dependence of our flat resolution on the choice of the fan $\Sigma$. We say that a complete fan $\Sigma'$ is a refinement of $\Sigma$ if for every $\sigma' \in \Sigma'$ there exists $\sigma \in \Sigma$ such that $\sigma' \subset \sigma$.

\begin{prop}
\begin{enumerate}[(a)]
    \item If $\Sigma'$ is a refinement of $\Sigma$, we have a canonical quasi-isomorphism of complexes $C^\bullet (\Sigma' ; \mathcal R) \longrightarrow C^\bullet(\Sigma;\mathcal R)$ commuting with the augmentation maps. 
    \item Any two complete fans $\Sigma_1,\Sigma_2$ admit a common refinement $\Sigma'$, yielding a roof of quasi-isomorphisms:
    \begin{equation}
        C^\bullet(\Sigma_1;\mathcal R) \longleftarrow C^\bullet(\Sigma'; \mathcal R) \longrightarrow C^\bullet(\Sigma_2 ;\mathcal R).
    \end{equation}
\end{enumerate}
\end{prop}
\begin{proof}
    (a) The required maps $F_p : C^p(\Sigma';\mathcal R) \longrightarrow C^p(\Sigma;\mathcal R)$ are given by
    \begin{equation}
        F_p(c')_{\sigma} = \sum_{\substack{\sigma' \in \Sigma'_p \\ \sigma' \subset \sigma}} [\sigma : \sigma'] c'_{\sigma'},
    \end{equation}
    where $[\sigma : \sigma'] = \pm 1$ compares the orientations of $\sigma,\sigma'$. We omit the simple verification that $F$ commutes with differentials and augmentation maps.

    (b) Take $\Sigma' = \{ \sigma_1 \cap \sigma_2 \mid \sigma_1 \in \Sigma_1, \ \sigma_2 \in \Sigma_2 \}$.
\end{proof}

\begin{remark}(Grothendieck topology)
One can regard finite unions of
cones as a Grothendieck site, where covers are given by
fans. That is, a~cover of
$X\subset \Lambda_{\mathbb R}$ is a fan whose union is $X$. In this language, complete fans are covers of $\Lambda_{\mathbb R}$. The assignment $\mathcal R$ is a precosheaf on this site. The whole setup is reminiscent of~\cite{artymowicz2024mathematical}.
\end{remark}

\subsection{Cup product} \label{sec:cup}

We introduce a cup product of cochains of the cellular flat resolution. To keep the notation manageable, we assume throughout the rest of the paper that the fan $\Sigma$ is simplicial. Then $\Sigma^\times$ defines a triangulation of $S^{d-1}$. We first review the usual cup product on $C^\bullet(\Sigma;\mathbb Z)$.

For $\sigma\in\Sigma$, let $V(\sigma)$ denote the set of rays of $\sigma$. Since $\Sigma$ is simplicial, every cone $\sigma\in\Sigma_p$ is uniquely determined by its $p+1$ rays. Choose a total ordering of all the rays of $\Sigma$, and orient each cone by the induced ordering of its rays. Given cochains~$f\in C^p(\Sigma;\mathbb Z)$ and $g\in C^q(\Sigma;\mathbb Z)$, their cup product
\begin{equation}
f\smile g\in C^{p+q}(\Sigma;\mathbb Z)
\end{equation}
is defined on a $(p+q)$-cell $\sigma\in\Sigma_{p+q}$ as follows. If the ordered rays of $\sigma$ are
\begin{equation}
\rho_0<\rho_1<\cdots<\rho_{p+q},
\end{equation}
let $\sigma_{\leq p} \in \Sigma_p$ be the face of $\sigma$ spanned by $\rho_0,\ldots,\rho_p$, and let $\sigma_{\geq p} \in \Sigma_{q}$ be the face spanned by $\rho_p,\ldots,\rho_{p+q}$. Then the $\sigma$-component of the cup product is given by:
\begin{equation}
(f \smile g)_\sigma
:=
f_{\sigma_{\leq p}} \cdot g_{\sigma_{\geq p}}.
\end{equation}

\begin{defn}[Cup product] \label{def:cup}
For $f\in C^p(\Sigma;\mathcal R)$ and $g\in C^q(\Sigma;\mathcal R)$, their cup product
\begin{equation}
f\smile g \in C^{p+q}(\Sigma;\mathcal R)
\end{equation}
is defined as follows. Let $\sigma\in\Sigma_{p+q}$ be spanned by rays $\rho_0<\cdots<\rho_{p+q}$, and let $\sigma_{\leq p}$ and $\sigma_{\geq p}$ be the faces of $\sigma$ spanned by $\rho_0,\ldots,\rho_p$ and $\rho_p,\ldots,\rho_{p+q}$, respectively. Then the $\sigma$-component of $f \smile g$ is defined by
\begin{equation}
(f\smile g)_\sigma
=
f_{\sigma_{\leq p}}
\cdot
g_{\sigma_{\geq p}}
\in \mathcal R(\sigma),
\end{equation}
where elements $f_{\sigma_{\leq p}}$ and $g_{\sigma_{\geq p}}$ are implicitly mapped to the $R$-algebra $\mathcal R(\sigma)$ via the face inclusions, and the product is taken in $\mathcal R(\sigma)$.
\end{defn}

\begin{exmp}\label{exmp:2d_resol}
In dimension $d=2$, let $\Lambda=\mathbb Z^2$, and let
\begin{equation}
\Sigma
=
\{\{0\}, \sigma_x, \sigma_y, \sigma_{\bar x}, \sigma_{\bar y},
\sigma_{x,y}, \sigma_{\bar x,y}, \sigma_{\bar x,\bar y}, \sigma_{x,\bar y}\}.
\end{equation}
Here $\sigma_{x,y}$, $\sigma_{\bar x,y}$, $\sigma_{\bar x,\bar y}$, and $\sigma_{x,\bar y}$ are the four quadrants, while $\sigma_x$, $\sigma_y$, $\sigma_{\bar x}$, and $\sigma_{\bar y}$ are the four coordinate-axis rays. We order the rays of $\Sigma$:
\begin{equation}
    \sigma_{\overline x} < \sigma_x < \sigma_{\overline y} < \sigma_y.
\end{equation}

 Choose an orientation of $S^1$, the induced orientations on the $1$-cells, and positive orientations for $0$-cells. The cellular cochain complex is 
\begin{equation} 
\mathcal R(\sigma_x)\oplus \mathcal  R(\sigma_y)\oplus \mathcal  R(\sigma_{\bar x})\oplus \mathcal  R(\sigma_{\bar y})
\xrightarrow{\delta^0}
\mathcal  R(\sigma_{x,y})\oplus \mathcal  R(\sigma_{\bar x,y})\oplus \mathcal  R(\sigma_{\bar x,\bar y})\oplus \mathcal  R(\sigma_{x,\bar y}),
\end{equation}
where
\begin{equation}
\delta^0(c_x,c_y,c_{\bar x},c_{\bar y})
=
(c_y-c_x,
c_y-c_{\bar x},
c_{\bar y}-c_{\bar x},
c_{\bar y}-c_x).
\end{equation}
The augmentation maps are given by
\begin{align}
\Delta(r)& =(r,r,r,r), \\
\int (c_{x,y},c_{\bar x,y},c_{\bar x,\bar y},c_{x,\bar y})
& =
c_{x,y}-c_{\bar x,y}+c_{\bar x,\bar y}-c_{x,\bar y}. \nonumber
\end{align}
The cup product of a $0$-cochain $c$ and a $1$-cochain $d$ is given by
\begin{equation}
    (c_x,c_y,c_{\bar x},c_{\bar y}) \smile (d_{x,y},d_{\bar x,y},d_{\bar x,\bar y},d_{x,\bar y}) = (c_x d_{x,y}, c_{\bar x} d_{\bar x,y}, c_{\bar x} d_{\bar x,\bar y}, c_x d_{x,\bar y}).
\end{equation}
\end{exmp}

\begin{prop}
   The cellular flat resolution $C^\bullet(\Sigma;\mathcal R)$, equipped with the cup product, is a~differential graded algebra (DGA) over $R$. If $d>1$, its cohomology modules are 
    \begin{equation}
        H^p(\Sigma;\mathcal R) = \begin{cases}
            R, & \text{for } p=0, \\
            \widehat R, & \text{for } p =d-1,\\
             0, & \text{otherwise}. \\
        \end{cases}
        \label{eq:dga_cohomology}
    \end{equation}
\end{prop}
\begin{proof}
Verification of the DGA axioms follows standard steps, while \eqref{eq:dga_cohomology} is a consequence of Theorem \ref{thm:cellular}.
\end{proof}

\subsection{Antipodal symmetry} \label{sec:antipodal}
From now on, we will always assume that $\Sigma$ is centrally symmetric (see Definition \ref{def:central_sym}). Then the antipodal map 
\begin{equation}
 a:\Lambda_{\mathbb R}\to \Lambda_{\mathbb R}, \qquad  a(v)=-v,   
\end{equation}
induces a simplicial transformation of $S^{d-1}$. We choose a total ordering of the rays $\Sigma_0$ such that the restriction of $a$ to any cell  $\sigma \in \Sigma^\times$ is order-preserving. As every cone is strongly convex, no~cell $\sigma$ can simultaneously contain a ray $\rho$ and its antipode $-\rho$. Therefore, we can achieve the required order-preserving property by grouping the rays into antipodal pairs and ordering the pairs:
\begin{equation}
    -\rho_1 < \rho_1 < - \rho_2 < \rho_2 < \dots < -\rho_m < \rho_m,
\label{eq:rho_ordering}
\end{equation}
where $\Sigma_0 = \{ \pm \rho_1, \dots, \pm \rho_m \}$.

The map induced by $a$ on integral chains is
\begin{equation}
a_* : C_p(\Sigma;\mathbb Z)\longrightarrow C_p(\Sigma;\mathbb Z),
\qquad
a_*[\sigma]=[-\sigma].
\end{equation}
The induced map on integral cochains is
\begin{equation}
a^* : C^p(\Sigma;\mathbb Z)\longrightarrow C^p(\Sigma;\mathbb Z),
\qquad
(a^* f)_\sigma=f_{-\sigma}.
\end{equation}

Finally, the antipodal map induces an involution on $\widehat R$, given by
\begin{equation}
\sum_{\lambda \in \Lambda} a_\lambda x^\lambda \mapsto
\overline{
\sum_{\lambda \in \Lambda} a_\lambda x^\lambda}
=
\sum_{\lambda \in \Lambda} a_\lambda x^{-\lambda}.
\end{equation}
This map restricts to an isomorphism
\(
\mathcal R(\sigma)
\xrightarrow{\sim}
\mathcal R(-\sigma)
\)
for each $\sigma \in \Sigma$. 

We define an action of the antipodal map on $\mathcal R$-valued cochains by
\begin{equation}
a^* : C^p(\Sigma;\mathcal R)\longrightarrow C^p(\Sigma;\mathcal R),
\qquad
(a^*c)_\sigma
=
\overline{c_{-\sigma}}
\in \mathcal R(\sigma). \label{eq:cochain_antipode}
\end{equation}
The maps $a_*,a^*$ define an action of $C_2=\{1, a\}$ on both the chain complex $C_\bullet(\Sigma;\mathbb Z)$, and on the cochain complex $C^\bullet(\Sigma;\mathcal R)$, commuting with the respective differentials. If $\gamma  \in C_p(\Sigma;\mathbb Z)$ and $c \in C^p(\Sigma;\mathcal R)$, then
\begin{equation}
    \langle a_* \gamma,c \rangle = \overline{ \langle \gamma, a^* c \rangle}.
    \label{eq:push_pull_conjugate}
\end{equation}

\begin{lem} \label{lem:a-naturality-cup}
Let $\Sigma$ be a centrally symmetric complete fan, with rays ordered as above. Then the antipodal involution $a$ acts compatibly with the cup product: for any
$f \in C^p(\Sigma;\mathcal R)$ and $g \in C^q(\Sigma;\mathcal R)$, one has
\begin{equation}
a^*(f \smile g) = (a^*f) \smile ( a^*g) .
\end{equation}
\end{lem}
\begin{proof}
    Let $\sigma\in \Sigma_{p+q}$, and let $\rho_0< \rho_1<\cdots< \rho_{p+q}$ be the ordered rays that span $\sigma.$ Then the opposite rays $( -\rho_i )_{0 \leq i \leq p+q}$ span $-\sigma$, and we have 
    \begin{equation}
         -\rho_0<-\rho_1<\cdots< -\rho_{p+q} 
    \end{equation}
    because the order is preserved by the antipodal map. Therefore, \begin{equation}
        -\sigma_{\leq p}=(-\sigma)_{\leq p}, \qquad    -\sigma_{\geq p}=(-\sigma)_{\geq p}.
    \end{equation} 
    The claim follows from the definition of $\smile$. 
\end{proof}

In summary, the $C_2$-action on $C^\bullet(\Sigma;\mathcal R)$ is by differential graded algebra automorphisms. The $C_2$-equivariant structure will play a crucial role in the forthcoming construction of quadratic forms.

  \section{Duality and defects}

Before we delve into the main content of this section, we state a few definitions and lemmas. Given a module $M$, we introduce an associated functor $\mathcal M$, mimicking Definition \ref{def:cal_R}, with
\begin{equation}
    \mathcal M(\sigma) = \mathcal R(\sigma) \otimes M.
\end{equation}
Let us recall that $\mathcal R(\sigma)$ was defined as a submodule of $\widehat R$. However, for non-flat modules $M$ the canonical map $\mathcal M(\sigma) \longrightarrow \widehat M$ need not be injective.

Given a complete fan $\Sigma$, we assign to $M$ the chain complex $T_\bullet \otimes M$, where 
\begin{equation}
T_p := \begin{cases}
 C^{d-p-1}(\Sigma; \mathcal R), & \text{for } 0 \leq p < d , \\
R, &  \text{for } p =d ,
\end{cases}
\end{equation}
as in \eqref{eq:T_complex}. Since $T_p$ are flat modules, this construction defines an exact functor from $R$-modules to chain complexes of $R$-modules. We also let
\begin{equation}
        C^p(\Sigma;\mathcal M) = C^p(\Sigma;\mathcal R) \otimes M \cong \prod_{\sigma \in \Sigma_p} \mathcal M(\sigma). 
\end{equation}
Therefore, we will view elements of $C^p(\Sigma;\mathcal M)$ as tuples $( m_\sigma )_{\sigma \in \Sigma_p}$, where $m_\sigma \in \mathcal M(\sigma)$. The differentials of the complex $C^\bullet(\Sigma;\mathcal M)$ and the augmentation maps are obtained from those introduced in Definitions \ref{def:cochains_with_coeff} and \ref{def:augmentation} by tensoring with $\operatorname{id}_M$. With a slight abuse of notation, we will abbreviate $\Delta \otimes \operatorname{id}_M $, $\delta^p \otimes \operatorname{id}_M$ and $\int \otimes \operatorname{id}_M$ to $\Delta, \delta^k$ and $\int$. We observe that $C^\bullet(\Sigma;\mathcal M)$ is a DG module over DGA $C^\bullet(\Sigma;\mathcal R)$.

We will use analogous notation for other modules, for instance $C^p(\Sigma;\mathcal N) = C^p(\Sigma;\mathcal R) \otimes N$ for a~module $N$.

The following corollary follows directly from Theorem \ref{thm:cellular}.
    \begin{cor} \label{cor:TM_homology}
    The complex $T_\bullet \otimes M$ has homology $\mathrm{Tor}^R_\bullet(\widehat R,M)$.  
\end{cor}

We state without proof a result in Section~5 of~\cite{ruba2024homological}. The version for $k=\mathbb F_p$ was proven in~\cite{haah2013commuting}.

\begin{lem}[Mobility lemma] \label{lem: mobile}
Let $\mathfrak a\subset R$ be a proper ideal. Then
\(
    \dim(R/\mathfrak a)=0
\)
if and only if there exists a positive integer $D$ such that
\(
    x_i^D-1\in \mathfrak a
\)
for every $i=1,\dots,d$.

In particular, taking $\mathfrak a=\mathrm{Ann}(M)$, this implies that finite length modules have finite cardinality.
\end{lem}

We note an important consequence below.  

\begin{lem} \label{lem:tensor_kills}
    If $M$ is an $R$-module of finite length, then $T_p \otimes M =0$ for $p \neq d$.
\end{lem}
\begin{proof}
    It is enough to show that $\mathcal R(\sigma) \otimes M =0$ for every $\sigma \in \Sigma^\times$. Let $y \in R$ be a monomial supported in $\sigma$. Then $y$ is invertible in $R$, and by Lemma~\ref{lem: mobile} there exists a positive integer $k$ such that $1-y^k$ annihilates $M$. Since $y$ is supported in $\sigma$, the element $1-y^k$ is invertible in $\mathcal R(\sigma)$, with inverse given by the geometric series $\sum_{j=0}^\infty y^{kj}$. 
\end{proof}

\subsection{Duality pairs}

Recall that the notion of a duality pair was introduced in Definition \ref{def:sesquilinear}.

\begin{exmp} \label{exmp:ev_pair}
    Let $F$ be a finitely generated free module. We have the duality pair $(F, F^*, \mathrm{ev})$, where $\mathrm{ev}$ is the evaluation map:
    \begin{equation}
        \mathrm{ev}(m, \varphi) = \varphi(m), \qquad m \in F, \ \  \varphi \in F^*.
    \end{equation} 
\end{exmp}
\begin{exmp} \label{exmp: symmetric}
    We are particularly interested in two types of duality pairs:
    \begin{itemize}
        \item Modules with an $\varepsilon$-hermitian form --- $(P,P,\Omega)$ with $\Omega$ satisfying $\Omega(p,p') = \varepsilon \, \overline{\Omega(p',p)}$ for some $\varepsilon \in \{ \pm 1 \}$.
        \item Pairs $(Q,L)$, where $P$ is a module with an $\varepsilon$-hermitian form, $L \subset P$ is a Lagrangian submodule, and $Q=P/L$. 
    \end{itemize} 
    We will return to these specific cases in dedicated Subsections \ref{sec:self-pairing} and \ref{sec:Lagrangian}.
    Both settings, with $\varepsilon =-1$, are essential in the theory of Pauli stabilizer codes with translation symmetry. 
\end{exmp}

The nondegeneracy assumption implies that $\Omega$ induces inclusions 
\begin{equation}
    \lambda_N : N \longrightarrow M^*, \qquad \lambda_M : M \longrightarrow  N^*,
\end{equation}
see Definition \ref{def:sesquilinear}.
Secondly, $\Omega$ extends canonically to pairings $\widehat M \times N \to \widehat R$ and $M \times \widehat N \to \widehat R$, inducing maps
\begin{equation}
    \widetilde \lambda_N : \widehat N \longrightarrow M^{\#}, \qquad \qquad \widetilde \lambda_M : \widehat M \longrightarrow N^{\#}.
\end{equation}

\begin{lem} \label{lem:surjectivity}
The maps $\widetilde \lambda_N, \widetilde \lambda_M$ are surjective. 
\end{lem}
\begin{proof}
We represent $\widetilde \lambda_N$ as the composition
\begin{equation}
     \widehat N \xrightarrow{\widehat \lambda_N} \widehat R \otimes \Hom(\overline M, R) \xrightarrow{\theta} \Hom(\overline M, \widehat R),
\end{equation}
where $\widehat \lambda_N = \operatorname{id}_{\widehat R} \otimes \lambda_N$, and $\theta$ is the canonical map
\(
    \theta( f \otimes \varphi )(m) = \varphi(m) f.
\)
We will show that both $\widehat \lambda_N$ and $\theta$ are surjective.

Firstly, the nondegeneracy of $\Omega$ implies that $\coker(\lambda_N)$ is a torsion module; see \cite[Proposition 10]{ruba2024homological} for the argument spelled out in a slightly less general setting. Since $\coker(\lambda_N)$ is finitely generated, it is annihilated by some regular element $f \in R$. We have $f \widehat R = \widehat R$ by the injectivity of $\widehat R$, and therefore $\widehat R \otimes \coker(\lambda_N) =0$. By the right-exactness of tensor products, we~have the exact sequence
\begin{equation}
    \widehat N \xrightarrow{\widehat \lambda_N} \widehat R \otimes M^* \longrightarrow \widehat R \otimes \coker(\lambda_N) \longrightarrow 0.
\end{equation}

Secondly, $M$ is torsion-free and admits an embedding $i : M \to F$ in a finitely generated free module $F$, see e.g.\ \cite[Lemma 8]{ruba2024homological}. We have the commutative diagram
\begin{equation}
    \begin{tikzcd}
\widehat{R} \otimes_R F^* \arrow[r, "\sim"] \arrow[d] & F^{\#} \arrow[d, "i^{\#}"] \\
\widehat{R} \otimes_R M^* \arrow[r, "\theta"]       & M^{\#}
\end{tikzcd},
\end{equation}
in which the top horizontal arrow is an isomorphism because $F$ is finitely generated free, and~$i^{\#}$ is surjective because $( -)^{\#}$ is an exact functor. The~composition 
\( \widehat{R} \otimes_R F^* \longrightarrow \widehat{R} \otimes_R M^* \xrightarrow{\theta} M^{\#} \) is surjective, so $\theta$ is surjective.
\end{proof}

We define two homological invariants associated with duality pairs, obtained by modifying the derived functors of $(-)^*$ and $\widehat{(-)}$. They measure the failure of $M$ and $N$ to be projective and of $\Omega$ to be unimodular. In the Introduction, we referred to these invariants as \textit{duality defects}, or~simply \textit{defects}.

\begin{defn} \label{def:aug_Hom}
Let $F_\bullet \xrightarrow{\sigma} M$ be a resolution of $M$ with finitely generated free modules. We augment the Hom complex $\Hom(\overline F_\bullet,R)$ with the map $N \xrightarrow{\sigma^* \lambda_N} F_0^*$:
\begin{equation}
    0 \longrightarrow N \xrightarrow{\sigma^* \lambda_N} F_0^* \longrightarrow F_1^* \longrightarrow F_2^* \longrightarrow \dots .
    \label{eq:augmented_Hom_complex}
\end{equation}
We denote the cohomology of the complex \eqref{eq:augmented_Hom_complex} by $\widetilde \Ext^\bullet(\overline M,R)$. We highlight that this construction depends on $\Omega$. 
\end{defn}

\begin{defn} \label{def:aug_tensor}
    Let $F_\bullet \xrightarrow{\sigma} M$ be a resolution of $M$ with finitely generated free modules. We augment the tensor product complex $\widehat R \otimes F_\bullet$ with the map $\widehat F_0 \xrightarrow{\widetilde \lambda_M \widehat \sigma } N^{\#}$:
    \begin{equation}
        \dots \longrightarrow \widehat F_2 \longrightarrow \widehat F_1 \longrightarrow \widehat F_0 \longrightarrow N^{\#} \longrightarrow 0.
        \label{eq:augmented_tensor_complex}
    \end{equation}
    We denote the homology of the complex \eqref{eq:augmented_tensor_complex} by $\widetilde \Tor_\bullet(\widehat R, M)$. 
\end{defn}

\begin{prop}
The complexes in \eqref{eq:augmented_Hom_complex} and \eqref{eq:augmented_tensor_complex} have the following (co)homology:
\begin{subequations} \label{eq:cohomology_eval}
\begin{align} 
    \widetilde \Ext^p(\overline M,R) &= \begin{cases}
        \Ext^p(\overline M,R), & p \neq 0, \\
        \coker( \lambda_N \colon N \longrightarrow M^*), & p =0.
    \end{cases} \\
        \widetilde \Tor_p(\widehat R,M) &= \begin{cases}
        \Tor_p (\widehat R,M), & p \neq 0, \\
        \ker(\widetilde \lambda_M \colon \widehat M \longrightarrow N^{\#}), & p =0 . \label{eq:cohomology_eval_Tor}
    \end{cases} 
\end{align}
\end{subequations}
\end{prop}
\begin{proof}
The case $p \neq 0$ is clear from the definition of extension and torsion groups. For $p=0$, we~have $\widetilde \Ext^0(\overline M,R)= \frac{\Ext^0(\overline M,R)}{\operatorname{im}(\sigma^* \lambda_N)} = \coker(\lambda_N)$, where we identify $M^*$ with $\operatorname{im}(\sigma^*) \subset F_0^*$. Finally, $\widetilde \Ext^{-1}(\overline M,R)=0$ because $\sigma^* \lambda_N$ is injective.

The proof of \eqref{eq:cohomology_eval_Tor} is analogous. Vanishing of $\widetilde \Tor_{-1}(\widehat R,M)$ follows from Lemma \ref{lem:surjectivity}.
\end{proof}

\begin{remark}
    The complexes \eqref{eq:augmented_Hom_complex} and \eqref{eq:augmented_tensor_complex} represent the mapping cones $N \to \operatorname{RHom}(\overline M,R)$ and $\widehat R \otimes^L M \to N^{\#}$ in the derived category of $R$-modules.

    In general, quasi-isomorphism classes of complexes contain more data than their (co)homology. Therefore, we view the complexes \eqref{eq:augmented_Hom_complex}, \eqref{eq:augmented_tensor_complex} as more fundamental than modules in \eqref{eq:cohomology_eval}. However, we will not pursue a treatment fully formulated using derived categories in this article.
\end{remark}

\begin{remark}[Stabilization]
The modules $\widetilde \Ext^p(\overline M,R)$ and $\widetilde \Tor_p(\widehat R,M) $ are invariant under replacing $(M,N)$ with $(M \oplus F, N \oplus F^*)$, where $F$ is a finite free module. In degree $p=0$, this property fails for the standard extension and torsion modules $\Ext^0(\overline M,R)= M^*$ and $\Tor_0(\widehat R,M) = \widehat M$.
\end{remark}

\begin{remark}
    Definitions \ref{def:aug_Hom} and \ref{def:aug_tensor} may be applied with the roles of $M$ and $N$ interchanged. In the cases of Example~\ref{exmp: symmetric}, the resulting invariants coincide, possibly up to a degree shift; see~Sections~\ref{sec:self-pairing} and~\ref{sec:Lagrangian}.
\end{remark}

The invariants introduced in Definitions \ref{def:aug_Hom} and \ref{def:aug_tensor} are dual to each other in the following sense.

\begin{prop} \label{prop:eval}
    The evaluation map \( (v, \varphi) \mapsto \widehat \varphi (v) \in \widehat R \), where $\varphi \in F_p^*$, $\widehat \varphi = \mathrm{id}_{\widehat R} \otimes \varphi$, and $v \in \widehat F_p$, descends to a~sesquilinear pairing
    \begin{equation}
   \operatorname{ev}_p \colon   \widetilde \Tor_p(\widehat R,M) \times  \widetilde \Ext^p( \overline M,R)   \longrightarrow \widehat R.
   \label{eq:evaluation_pairing}
    \end{equation}
    This pairing induces canonical isomorphisms, for all $p$,
    \begin{equation}
        \widetilde \Tor_p(\widehat R,M) \cong \widetilde \Ext^p(\overline M,R)^{\#}.
        \label{eq:Tor_Ext_dual_ev}
    \end{equation}
\end{prop}
\begin{proof}
For $p >0$ this is covered by Proposition \ref{prop:Ext-Tor-duality}. It remains to consider the case $p=0$. We have the short exact sequence
\begin{equation}
    0 \longrightarrow N \xrightarrow{\lambda_N} M^* \longrightarrow \widetilde \Ext^0 (\overline M,R) \longrightarrow 0,
\end{equation}
so, by the exactness of $(-)^{\#}$:
\begin{equation}
\begin{split}
    \widetilde \Ext^0 (\overline M,R)^{\#} &\cong \{ \varphi \in (M^*)^{\#}  \mid \left. \varphi \right|_{\operatorname{im}(\lambda_N)} =0 \} \\
    & \cong \{ v \in \widehat M \mid \widetilde \lambda_M(v) = 0 \} \cong \widetilde \Tor_0(\widehat R,M),
\end{split}
\end{equation}
where in the second isomorphism we used Proposition \ref{prop:Ext-Tor-duality}. 
\end{proof}

\begin{cor} \label{cor:equal_Krull}
The modules $\widetilde \Ext^p(\overline M,R)$ and $ \widetilde \Tor_p(\widehat R,M)$ have equal Krull dimensions for every $p$. Moreover, for every $p$ the following conditions are equivalent:
\begin{enumerate}
    \item $\widetilde \Ext^p(\overline M,R)$ has finite length,
    \item  $\widetilde \Tor_p(\widehat R,M)$ has finite length,
    \item  $\widetilde \Tor_p(\widehat R,M)$ is finitely generated.
\end{enumerate} 
\end{cor}

\begin{proof}
    Follows from Proposition \ref{prop:eval} and the discussion in Appendix~\ref{sec:dual_dim}. 
\end{proof}

\begin{prop}
We have $\widetilde \Ext^p(\overline M,R) = 0$ and $\widetilde \Tor_p(\widehat R,M) =0$ for $p \geq d$, and the following upper bound on the Krull dimensions:
\begin{equation}
\dim \widetilde \Tor_p(\widehat R,M)=    \dim \widetilde \Ext^p(\overline M,R) \leq d-p-1.
\end{equation}
\end{prop}
\begin{proof}
    By Corollary \ref{cor:equal_Krull}, it is enough to consider the $\widetilde \Ext$ modules. The case \(p>0\) is proved in \cite{ruba2024homological} using the standard dimension bound for Ext modules; see \cite[Corollary~3.5.11(c)]{Bruns_Herzog}. For $p=0$, we note that $\widetilde \Ext^0(\overline M,R)$ is a torsion module, so its Krull dimension is less than $d$. 
\end{proof}

\subsection{Pairings from cup products}

We will now define a pairing of the $\widetilde \Tor$ modules using the cellular flat resolution. Let us begin by recalling that $\Tor_\bullet(\widehat R,M)$ can be computed from the complex
\begin{equation}
    0 \longrightarrow M \longrightarrow C^0(\Sigma;\mathcal M) \longrightarrow \dots \longrightarrow C^{d-1}(\Sigma;\mathcal M) \longrightarrow  0.
\end{equation}
Moreover, we have the surjection $\int \colon C^{d-1}(\Sigma;\mathcal M) \longrightarrow \widehat M$ whose kernel is the module of coboundaries in $C^{d-1}(\Sigma;\mathcal M)$. 

\begin{defn} \label{def:smile_Omega_product}
Let $(M,N,\Omega)$ be a duality pair. We define the sesquilinear pairing
\begin{equation}
 C^p(\Sigma; \mathcal M) \times C^q(\Sigma;\mathcal N) \longrightarrow C^{p+q} (\Sigma;\mathcal R), \qquad (\alpha,\beta) \mapsto a^* \alpha \smile^\Omega \beta,
 \label{eq:Omega_cup}
 \end{equation} 
by first specifying it on simple tensors. If $\alpha = f \otimes m$ and $\beta = g \otimes n$, with
\begin{equation}
    f \in C^p(\Sigma;\mathcal R), \qquad g \in C^q(\Sigma;\mathcal R), \qquad m \in M, \qquad n \in N,
\end{equation}
we set (cf.\ Definition \ref{def:cup} and \eqref{eq:cochain_antipode}):
\begin{equation}
    a^* \alpha \smile^\Omega \beta = \Omega(m,n) \, a^* f \smile g.
\end{equation}

Equivalently, this product may be described component-wise. Let $\sigma\in\Sigma_{p+q}$ be spanned by rays $\rho_0<\cdots<\rho_{p+q}$, and let $-\sigma_{\leq p}$ and $\sigma_{\geq p}$ be the faces of $\sigma$ spanned by $-\rho_0,\ldots,-\rho_p$ and $\rho_p,\ldots,\rho_{p+q}$, respectively. The $\sigma$-component of $a^* \alpha \smile^\Omega \beta$ is defined by
\begin{equation}
    (a^* \alpha \smile^\Omega \beta)_\sigma = \Omega(\alpha_{- \sigma_{\leq p}}, \beta_{\sigma_{\geq p}}) ,
\end{equation}
using the natural extension of $\Omega$ to a pairing $\mathcal M(- \sigma_{\leq p}) \times \mathcal N(\sigma_{\geq p}) \to \mathcal R(\sigma)$. 
\end{defn}

\begin{defn} \label{def:smile_Omega_pairing}
We introduce sesquilinear pairings
\begin{equation} 
C^{d-p-1}(\Sigma;\mathcal M) \times C^{p}(\Sigma;\mathcal N) \longrightarrow \widehat R, \qquad (\alpha,\beta) \mapsto \int a^* \alpha \smile^\Omega \beta
\label{eq:cup_pairing}
\end{equation}
using the product \eqref{eq:Omega_cup} and the augmentation map $\int$, cf.\ \eqref{eq:integral}.
\end{defn}

The following proposition is a step towards proving Theorem \ref{thm:intro-poincare}.

\begin{prop} \label{prop:cup_pairing}
\begin{enumerate}[(a)]
    \item The map \eqref{eq:cup_pairing} descends to a sesquilinear pairing
    \begin{equation}
    \operatorname{Br}_p \colon \widetilde \Tor_{p}(\widehat R,M) \times \widetilde \Tor_{d-p-1} (\widehat R,N) \longrightarrow \widehat R.
    \label{eq:Tor_Tor_pairing}
    \end{equation}
    \item If every $\widetilde \Ext^\bullet(\overline M,R)$ module has finite length, then $\operatorname{Br}_p$ are perfect pairings: they induce isomorphisms
    \begin{equation}
        \widetilde \Tor_{d-p-1}(\widehat R,N) \cong \widetilde \Tor_p(\widehat R,M)^{\#}. 
    \end{equation}
\end{enumerate}
\end{prop}

\begin{proof}[Proof of Proposition~\ref{prop:cup_pairing} (a)]
In this proof, we will denote the cohomology of $C^\bullet(\Sigma;\mathcal M)$ and $C^\bullet(\Sigma;\mathcal N)$ by $H^\bullet(\Sigma;\mathcal M)$ and $H^\bullet(\Sigma;\mathcal N)$.

The product $\smile^\Omega$ satisfies the Leibniz rule for the differential $\delta$, so $\eqref{eq:cup_pairing}$ descends to a pairing
\begin{equation}
    H^{d-p-1}(\Sigma;\mathcal M) \times H^p(\Sigma;\mathcal N) \longrightarrow \widehat R. \label{eq:cellular_cohomology_pairing}
\end{equation}
If $p \not \in \{ 0 , d-1 \}$, we have the isomorphisms
\begin{subequations}
    \begin{align}
        H^{d-p-1}(\Sigma;\mathcal M) & \cong \widetilde \Tor_{p}(\widehat R, M), \\ 
        H^{p}(\Sigma;\mathcal N) & \cong \widetilde \Tor_{d-p-1}(\widehat R, N),
    \end{align}
\end{subequations}
which gives $\operatorname{Br}_p$. 

Let us now assume that $d \neq 1$, and take $p=d-1$ (by~the symmetry between $M$ and $N$, $p=0$ can be handled the same way). Then \eqref{eq:cellular_cohomology_pairing} reduces to
\begin{equation}
   H^0(\Sigma;\mathcal M) \times \widehat N \longrightarrow \widehat R.
   \label{eq:H0_widehatN}
\end{equation}
We will now show that this pairing vanishes if the first argument is in $\operatorname{im}(\Delta)$ and the second argument is in $\widetilde \Tor_0(\widehat R,N) \subset \widehat N$. We consider \eqref{eq:cup_pairing} with $p=d-1$, $\alpha = \Delta(m)$ for some $m \in M$, and $\beta $ satisfying $\Omega(\cdot,\int \beta)=0$. Following Definition \ref{def:smile_Omega_product}, we evaluate
\begin{equation}
    \int a^* \alpha \smile^\Omega \beta = \Omega \left(m , \int \beta \right)=0.
\end{equation}
Since $H^0(\Sigma;\mathcal M) / \operatorname{im}(\Delta) \cong \Tor_{d-1}(\widehat R, M) = \widetilde \Tor_{d-1}(\widehat R, M)$, we conclude that \eqref{eq:H0_widehatN} restricts to a~pairing
\begin{equation}
    \widetilde \Tor_{d-1}(\widehat R,M) \times \widetilde \Tor_0(\widehat R,N) \longrightarrow \widehat R.
\end{equation}

One checks that with small adjustments the above argument applies also in the case $d=1$.
\end{proof}

We postpone the proof of Proposition~\ref{prop:cup_pairing} (b) to the end of Section \ref{sec:double_complex}.

\subsection{Poincar\'e duality} \label{sec:double_complex}

In this section, we complete the proof of Theorem \ref{thm:intro-poincare}.

Let $F_\bullet \xrightarrow{\sigma} M$ and $G_\bullet \xrightarrow{\tau} N$ be resolutions by finite free modules. We~apply the functor $({-})^*$ to $F_\bullet$ and join the two complexes:
\begin{equation}
 \mathbb F_\Omega\colon \quad   \dots \longrightarrow G_2 \longrightarrow G_1 \longrightarrow G_0 \xrightarrow{\sigma^* \lambda_N \tau} F_0^* \longrightarrow F_1^* \longrightarrow F_2^* \longrightarrow \dots
    \label{eq:double_complex}
\end{equation}
using the map $G_0 \longrightarrow F_0^*$ given as the composition 
\begin{equation}
    G_0 \xrightarrow{\tau} N \xrightarrow{\lambda_N} M^* \xrightarrow{\sigma^*} F_0^*. 
\end{equation}
Furthermore, let us tensor the complex $\mathbb F_\Omega$ in \eqref{eq:double_complex} with the exact sequence \eqref{eq:augmented_exact_complex}, thereby obtaining the following double complex:
\begin{equation}
\begin{tikzcd}[row sep=1.5em, column sep=1.5em, font=\footnotesize]
    & 0 \arrow[d] & 0 \arrow[d] & 0 \arrow[d] & 0 \arrow[d] & \\
    \cdots \arrow[r] 
    & G_1 \arrow[r] \arrow[d, "\Delta"] 
    & G_0 \arrow[r] \arrow[d, "\Delta"] 
    & F_0^* \arrow[r] \arrow[d, "\Delta"] 
    & F_1^* \arrow[r] \arrow[d, "\Delta"] 
    & \cdots \\
    \cdots \arrow[r] 
    & T_{d-1} \otimes G_1 \arrow[r] \arrow[d, "\delta^0"] 
    & T_{d-1} \otimes G_0 \arrow[r] \arrow[d, "\delta^0"] 
    & T_{d-1} \otimes F_0^* \arrow[r] \arrow[d, "\delta^0"] 
    & T_{d-1} \otimes F_1^* \arrow[r] \arrow[d, "\delta^0"] 
    & \cdots \\
    & \vdots \arrow[d, "\delta^{d-2}"] 
    & \vdots \arrow[d, "\delta^{d-2}"] 
    & \vdots \arrow[d, "\delta^{d-2}"] 
    & \vdots \arrow[d, "\delta^{d-2}"] 
    & \\
    \cdots \arrow[r] 
    & T_0 \otimes G_1 \arrow[r] \arrow[d, "\int"] 
    & T_0 \otimes G_0 \arrow[r] \arrow[d, "\int"] 
    & T_0 \otimes F_0^* \arrow[r] \arrow[d, "\int"] 
    & T_0 \otimes F_1^* \arrow[r] \arrow[d, "\int"] 
    & \cdots \\
    \cdots \arrow[r] 
    & \widehat G_1 \arrow[r] \arrow[d] 
    & \widehat G_0 \arrow[r] \arrow[d] 
    &  F_0^{\#} \arrow[r] \arrow[d] 
    &  F_1^{\#} \arrow[r] \arrow[d] 
    & \cdots \\
    & 0 & 0 & 0 & 0 &
\end{tikzcd},
\label{eq:tensor_double_complex}
\end{equation}
where in the bottom row we identified $\widehat R \otimes F_i^* = F_i^{\#}$, cf.\ Lemma \ref{lem:star_hat_sharp}.

\begin{lem} \label{lem:double_complex_rows_and_cols}
    The double complex \eqref{eq:tensor_double_complex} has exact columns. The horizontal homology of its nonzero rows is given as follows.
    \begin{enumerate}[(a)]
        \item The top row computes $\widetilde \Ext^\bullet(\overline M,R)$ in the columns corresponding to~$F_\bullet^*$, and is exact in the $G_\bullet$ columns.
        \item The intermediate rows (tensored with $T_p$ with $0 \leq p \leq d-1$) compute $T_p \otimes \widetilde \Ext^\bullet(\overline M,R)$, also residing in the $F_\bullet^*$ columns.
        \item The bottom row computes $\widetilde \Tor_\bullet(\widehat R,N)$ in the $G_\bullet$ columns, and is exact in the $F_\bullet^*$ columns.
    \end{enumerate}
\end{lem}
\begin{proof}
(a) We have the commutative diagram with exact rows
\begin{equation}
\begin{tikzcd}
    \dots \arrow[r] & G_1 \arrow[r] \arrow[d] & G_0 \arrow[r, "\sigma^* \lambda_N \tau"] \arrow[d, "\tau"] & F_0^* \arrow[r] \arrow[d, "\mathrm{id}"] & F_1^* \arrow[r] \arrow[d, "\mathrm{id}"] & \dots \\
    \dots \arrow[r] & 0 \arrow[r]             & N \arrow[r, "\sigma^* \lambda_N"]                                                & F_0^* \arrow[r]                          & F_1^* \arrow[r]                          & \dots
\end{tikzcd}
\end{equation}
whose vertical arrows provide a quasi-isomorphism from the top row of the double complex \eqref{eq:tensor_double_complex} to the augmented Hom complex \eqref{eq:augmented_Hom_complex}.

(b) follows from (a) because $T_p$ are flat modules.

(c) As in the proof of (a), we have the commutative diagram with exact rows
\begin{equation}
\begin{tikzcd}
    \dots \arrow[r] & \widehat{G}_1 \arrow[r]  & \widehat{G}_0 \arrow[r, "\sigma^\# \widetilde\lambda_N \widehat{\tau}"]  & F_0^\# \arrow[r]  & F_1^\# \arrow[r]  & \dots \\
    \dots \arrow[r] & \widehat{G}_1 \arrow[r] \arrow[u, "\mathrm{id}"]                                   & \widehat{G}_0 \arrow[r, "\widetilde \lambda_N \widehat{\tau}"]  \arrow[u, "\mathrm{id}"]                                             & M^\# \arrow[r] \arrow[u, "\sigma^\#"]                                           & 0 \arrow[r]   \arrow[u]                   & \dots
\end{tikzcd}
\end{equation}
whose vertical arrows provide a quasi-isomorphism from the augmented tensor complex \eqref{eq:augmented_tensor_complex} to the bottom row of the double complex \eqref{eq:tensor_double_complex}.
\end{proof}

\begin{prop}\label{prop:kappa_map}
\begin{enumerate}[(a)]
    \item There exist canonical homomorphisms
\begin{equation}
\kappa_p \colon \widetilde \Tor_p(\widehat R,N) \longrightarrow \widetilde \Ext^{d-p-1}(\overline M,R).
\label{eq:Tor_Ext_maps}
\end{equation}
If every $\widetilde \Ext^q(\overline M,R)$ has finite length, then the maps $\kappa_p$ are isomorphisms.
\item The pairing $\operatorname{Br}_p$, defined in Proposition~\ref{prop:cup_pairing}, can be expressed as
\begin{equation}
 \operatorname{Br}_p =c_{d,p} \, \operatorname{ev}_p \circ (\operatorname{id} \times \kappa_{d-p-1}),
 \label{eq:Br_factorization}
\end{equation}
where $\operatorname{ev}_p$ is the pairing defined in Proposition \ref{prop:eval}, and $c_{d,p} \in \{ \pm 1 \}$ depends only on $d,p$. 
\end{enumerate}
\end{prop}
Propositions \ref{prop:cup_pairing} and \ref{prop:kappa_map} combined establish Theorem \ref{thm:intro-poincare}.
\begin{proof}
(a) We define $\kappa_p$ to be the edge-to-edge maps, defined in Proposition \ref{prop:edge-to-edge} in Appendix~\ref{app:double_complex}, associated to the double complex \eqref{eq:tensor_double_complex}. 

If every $\widetilde \Ext^{\bullet}(\overline M,R)$ module has finite length, then by Lemma \ref{lem:tensor_kills} it is killed by tensoring with $T_p$ in Lemma \ref{lem:double_complex_rows_and_cols} (b). Then all the rows of \eqref{eq:tensor_double_complex} except for the top and bottom rows are exact, and $\kappa$ maps are isomorphisms by Proposition \ref{prop:edge-to-edge}.

(b) Consider a class $h \in \widetilde \Tor_{d-p-1}(\widehat R,N)$ represented by $\beta \in C^p (\Sigma;\mathcal N)$ with $0 \leq p \leq d-1$. If~$p \neq d-1$, then $\beta$ is required to satisfy the cocycle condition $\delta \beta =0$. If $p=d-1$, the cocycle condition is replaced by the requirement that $\Omega \left( \cdot ,\int \beta \right) =0$. 

We choose a lift $\beta_p \in C^p(\Sigma;\mathcal R) \otimes G_0$; that is, we require that 
\begin{equation}
(    \operatorname{id}_{C^p(\Sigma;\mathcal R)} \otimes \tau)(\beta_p) = \beta, 
\end{equation}
where $\tau$ is the augmentation map of the resolution $G_\bullet \xrightarrow{\tau} N$. In the rest of the proof, we will abbreviate this and other similar formulas and write $\tau(\beta_p)=\beta$, leaving identity maps on modules $C^p(\Sigma;\mathcal R)$ implicit. 

We can find elements $\beta_{p+1},\dots, \beta_{d-1}$, with $\beta_q \in C^q(\Sigma;\mathcal R) \otimes G_{q-p}$, such that
\begin{equation}
    (-1)^{p-q} d_h \beta_q + \delta \beta_{q-1} =0, \qquad p+1 \leq q \leq d-1.
    \label{eq:double_complex_recursion}
\end{equation}
Here $d_h$ and $\delta$ are the horizontal and vertical differentials of the double complex \eqref{eq:tensor_double_complex}. To verify the existence of the $\beta_q$ satisfying \eqref{eq:double_complex_recursion}, it suffices to note that the rows of \eqref{eq:tensor_double_complex} are exact in the relevant part of the double complex. Comparing this construction with the standard balancing Tor arguments, we see that the cycle $\int \beta_{d-1} \in \widehat G_{d-p-1}$ represents the class $h$, where now we compute the torsion module by resolving~$N$. 

By exactness of columns of \eqref{eq:tensor_double_complex}, we can find $\beta_0,\dots,\beta_{p-1}$, with $\beta_q \in C^q(\Sigma;\mathcal R) \otimes F_{p-q-1}^*$ and $\beta_{-1} \in F_p^*$, such that relations \eqref{eq:double_complex_recursion} hold also for $1 \leq q \leq p$, and in addition
\begin{equation}
    (-1)^p d_h \beta_0 + \Delta \beta_{-1} =0. 
\end{equation}
Then $\beta_{-1}$ is a horizontal cocycle representing the class $\kappa(h) \in \widetilde \Ext^p(\overline M, R)$, cf.~the proof of (a) and Appendix \ref{app:double_complex}. 

We proceed analogously for a class $h' \in \widetilde \Tor_{p}(\widehat R,M)$ represented by a cocycle $\alpha \in C^{d-p-1}(\Sigma;\mathcal M)$ ($\int \alpha$ orthogonal to $N$ if $p=0$). In this case we consider a version of the double complex \eqref{eq:tensor_double_complex} with the roles of $M$ and $N$ (and thus also of $F_\bullet$ and $G_\bullet$) exchanged. Let us note that switching the roles of $M$ and $N$ in \eqref{eq:double_complex} is equivalent, up to a degree shift, to applying the functor $(-)^*$ to the complex. Therefore, we denote the new horizontal differentials by $d_h^*$.

Arguing as for $h$ and $\beta$, we obtain elements $\alpha_{-1},\dots,\alpha_{d-1}$, where:
\begin{equation}
    \alpha_q \in \begin{cases}
        G_{d-p-1}^*, & q=-1, \\
        C^q(\Sigma;\mathcal R) \otimes G_{d-p-q-2}^*, & 0 \leq q \leq d-p-2, \\
        C^q(\Sigma;\mathcal R) \otimes F_{q+p-d+1} , & d-p-1 \leq q \leq d-1,
    \end{cases}
    \label{eq:alpha_recursions}
\end{equation}
and we have the equations:
\begin{subequations}
\begin{align}
    (-1)^{d-p-1} d_h^* \alpha_0 + \Delta \alpha_{-1} & = 0 , \\
    (-1)^{d-p-q-1} d_h^* \alpha_q + \delta \alpha_{q-1}   & = 0 , \qquad  1 \leq q \leq d-1, \label{eq:cocycle_2}\\
    \sigma(\alpha_{d-p-1}) &= \alpha. 
\end{align}
\end{subequations}
Then $\int \alpha_{d-1} \in \widehat F_p$ represents $h'$, and $\alpha_{-1}$ represents $\kappa(h')$. 

Let us now observe that
\begin{equation}
    \int a^* \alpha \smile^\Omega \beta = \int a^* \alpha_{d-1-p} \smile^{\mathrm{ev}} (\sigma^* \lambda_N \tau)(\beta_p),
\end{equation}
where $\smile^{\mathrm{ev}}$ is the product as in Definition \ref{def:smile_Omega_product} for the duality pair $(F_0,F_0^*)$ (see Example \ref{exmp:ev_pair}). The map $\sigma^* \lambda_N \tau$ is one of the differentials in the complex $\mathbb F_\Omega$ in \eqref{eq:double_complex}, so we can use relation \eqref{eq:double_complex_recursion} (with $q=p$) to express the above as
\begin{equation}
       \int a^* \alpha \smile^\Omega \beta = - \int a^* \alpha_{d-p-1} \smile^{\mathrm{ev}} \delta \beta_{p-1}. 
\end{equation}
Next, we integrate by parts and use \eqref{eq:cocycle_2} with $q=d-p$:
\begin{align}
    \int a^* \alpha \smile^\Omega \beta &= (-1)^{d-p-1} \int a^* \delta \alpha_{d-p-1} \smile^{\mathrm{ev}} \beta_{p-1} \\
    & = (-1)^{d-p-1} \int a^* d_h^* \alpha_{d-p-1} \smile^{\mathrm{ev}} \beta_{p-1} \nonumber \\ 
    & = (-1)^{d-p-1} \int a^* \alpha_{d-p-1} \smile^{\mathrm{ev}} d_h \beta_{p-1}, \nonumber 
    \end{align}
    where in the final expression we consider the duality pair $(F_1,F_1^*)$. Repeating this step $p$ times in total, we arrive at the formula
    \begin{align}
         \int a^* \alpha \smile^\Omega \beta &= (-1)^{p(d-p) + \frac{p(p+1)}{2}} \int a^* \alpha_{d-1} \smile^{\mathrm{ev}} d_h \beta_0 \\
         & = (-1)^{p(d-p) + \frac{p(p+1)}{2}+ p+1} \int a^* \alpha_{d-1} \smile^{\mathrm{ev}} \Delta \beta_{-1} \nonumber .
    \end{align} 
    This integral can now be evaluated from definition:
    \begin{align}
        \int a^* \alpha_{d-1} \smile^{\mathrm{ev}} \Delta \beta_{-1} &= \sum_{\sigma \in \Sigma_{d-1}} [S^{d-1}:\sigma] \operatorname{ev}((\alpha_{d-1})_{-\sigma},\beta_{-1}) \\
        & = (-1)^d \operatorname{ev} \left( \int \alpha_{d-1}, \beta_{-1} \right) = (-1)^d \operatorname{ev}(h', \kappa(h)), \nonumber
    \end{align}
    where $(-1)^d$ is present because $[S^{d-1}:-\sigma] = (-1)^d [S^{d-1}:\sigma]$. This establishes \eqref{eq:Br_factorization} with
    \begin{equation}
         c_{d,p}=(-1)^{p(d-p) + \frac{p(p+1)}{2} + p+1 + d}. 
    \end{equation}
\end{proof}

\begin{proof}[Proof of Proposition~\ref{prop:cup_pairing} (b)]
Immediate consequence of Proposition \ref{prop:eval} and Proposition \ref{prop:kappa_map}.
\end{proof}

\begin{cor}
Suppose that every $\widetilde \Ext^p(\overline M,R)$ module has finite length. There exist canonical and perfect sesquilinear pairings
    \begin{equation}
        \widetilde \Ext^{d-p-1}(\overline N,R) \times \widetilde \Ext^{p}(\overline M,R) \longrightarrow \widehat R, 
    \end{equation}
    which induce isomorphisms
    \begin{equation}
        \widetilde \Ext^{d-p-1}(\overline N,R) \cong \widetilde \Ext^{p}(\overline M,R)^{\#}, \quad \widetilde \Ext^{p}(\overline M,R) \cong \widetilde \Ext^{d-p-1}(\overline N,R)^{\#}.
    \end{equation}
\end{cor}
\begin{proof}
We transport pairings $\operatorname{Br}$ from Theorem \ref{prop:cup_pairing} via the isomorphisms $\kappa$ from Theorem \ref{prop:kappa_map}.
\end{proof}

\subsection{Modules with self-pairings}
\label{sec:self-pairing}

In this subsection and the next, we specialize the preceding results to the two classes of duality pairs introduced in Example~\ref{exmp: symmetric}. No new results are proved; rather, we record the resulting self-dualities in a form suited to the quadratic constructions of the next section.

Let $P$ be a finitely generated $R$-module equipped with a nondegenerate $\varepsilon$-hermitian form
\begin{equation}
\Omega \colon P \times P \longrightarrow R,
\qquad
\Omega(p,q)=\varepsilon\,\overline{\Omega(q,p)},
\qquad
\varepsilon\in\{\pm1\}.
\label{eq:Omega_epsilon_hermitian}
\end{equation}
Then $(P,P,\Omega)$ is a duality pair. We write
\begin{equation}
\lambda_P \colon P\longrightarrow P^*,
\qquad
p\longmapsto\Omega(\mathord\cdot,p),
\label{eq:self-duality-map}
\end{equation}
for the induced embedding and
\(
\widetilde\lambda_P \colon\widehat P\longrightarrow P^\#
\)
for its extension, which is surjective by Lemma~\ref{lem:surjectivity}. Thus
\begin{equation}
\widetilde\Ext^0(\overline P,R)=\coker(\lambda_P),
\qquad
\widetilde\Tor_0(\widehat R,P)=\ker(\widetilde\lambda_P),
\label{eq:self-degree-zero-invariants}
\end{equation}
while in positive degrees the modified invariants agree with the usual $\Ext$ and $\Tor$ modules.

     In our earlier work \cite{ruba2024homological}, we called sesquilinear forms $\Omega$ with an alternating scalar part \emph{quasi-symplectic}. We chose this terminology because we viewed the scalar part of $\Omega$ as more fundamental than $\Omega$ itself.

\begin{remark}
For a general sesquilinear form, one also has the map
\begin{equation}
P\longrightarrow P^*,
\qquad
q\longmapsto\overline{\Omega(q,\mathord\cdot)}.
\end{equation}
For the $\varepsilon$-hermitian form~\eqref{eq:Omega_epsilon_hermitian}, this map is $\varepsilon\lambda_P$ and therefore contains no additional information.
\end{remark}

Let $F_\bullet\xrightarrow{\sigma}P$ be a resolution by finitely generated free modules. Using this resolution for both entries of the duality pair, the two-sided complex~\eqref{eq:double_complex} takes the self-dual form
\begin{equation}
\mathbb F_\Omega\colon\quad
\cdots\longrightarrow F_2\longrightarrow F_1\longrightarrow F_0
\xrightarrow{\sigma^*\lambda_P\sigma}
F_0^*\longrightarrow F_1^*\longrightarrow F_2^*\longrightarrow\cdots.
\label{eq:self-two-sided-complex}
\end{equation}
Thus a single resolution supplies both halves of the construction. The double complex defining the edge maps is simply the specialization of~\eqref{eq:tensor_double_complex} obtained by taking $M=N=P$ and $G_\bullet=F_\bullet$.

\begin{cor}[Complementary-degree self-duality]
\label{cor:self-pairing-duality}
The cup-product construction defines, for every $0\leq p\leq d-1$, a canonical sesquilinear pairing
\begin{equation}
\operatorname{Br}_p\colon
\widetilde\Tor_p(\widehat R,P)
\times
\widetilde\Tor_{d-p-1}(\widehat R,P)
\longrightarrow\widehat R.
\label{eq:self-braiding-pairing}
\end{equation}
If $u=[\alpha]$ and $v=[\beta]$ are represented in the cellular model by
\begin{equation}
\alpha\in C^{d-p-1}(\Sigma;\mathcal P),
\qquad
\beta\in C^p(\Sigma;\mathcal P),
\end{equation}
then
\begin{equation}
\operatorname{Br}_p(u,v)
=
\int a^*\alpha\smile^\Omega\beta.
\label{eq:self-cup-pairing}
\end{equation}
Equivalently, in terms of the edge map and evaluation pairing,
\begin{equation}
\operatorname{Br}_p(u,v)
=
c_{d,p}\,\operatorname{ev}_p\bigl(u,\kappa_{d-p-1}(v)\bigr),
\label{eq:self-homological-pairing}
\end{equation}
where
\(
\kappa_{d-p-1}\colon
\widetilde\Tor_{d-p-1}(\widehat R,P)
\to
\widetilde\Ext^p(\overline P,R)
\) and $c_{d,p}\in\{\pm1\}$ are as in Proposition~\ref{prop:kappa_map}.

Suppose, in addition, that every $\widetilde\Ext^p(\overline P,R)$ module has finite length. Then all maps $\kappa_p$ are isomorphisms and the pairings~\eqref{eq:self-braiding-pairing} are perfect. In particular, they induce canonical isomorphisms
\begin{subequations}
\begin{align}
\widetilde\Tor_{d-p-1}(\widehat R,P)
& \cong
\widetilde\Tor_p(\widehat R,P)^\#, \\
\widetilde\Ext^{d-p-1}(\overline P,R)
& \cong
\widetilde\Ext^p(\overline P,R)^\#.
\label{eq:self-complementary-isomorphisms}
\end{align}
\end{subequations}
\end{cor}

\begin{proof}
Apply Propositions~\ref{prop:cup_pairing} and~\ref{prop:kappa_map} to the duality pair $(P,P,\Omega)$. The final assertion also uses Proposition~\ref{prop:eval}.
\end{proof}

\begin{remark}[Middle-degree self-pairing]
\label{rem:self-pairing-middle-degree}
When $d$ is odd,
\eqref{eq:self-braiding-pairing} specializes to a canonical self-pairing
\begin{equation}
\operatorname{Br}_{\frac{d-1}{2}}\colon
\widetilde\Tor_{\frac{d-1}{2}}(\widehat R,P)
\times
\widetilde\Tor_{\frac{d-1}{2}}(\widehat R,P)
\longrightarrow \widehat R.
\label{eq:self-middle-degree-pairing}
\end{equation}
Properties of this self-pairing are discussed in Section \ref{sec:quadratic_self_pair}.
\end{remark}

\subsection{Lagrangian submodules}
\label{sec:Lagrangian}

We now specialize to the second class of duality pairs in Example~\ref{exmp: symmetric}. These are precisely the duality pairs arising from translation-invariant Pauli stabilizer codes. The resulting pairing coincides with the braiding pairing defined in Section~6 of~\cite{ruba2024homological}. In~lattice dimension~2, its perfectness was established in~\cite{ruba2025witt} by a different method, namely the bulk--boundary correspondence. This case is particularly relevant to physics, where the pairing encodes the braiding of abelian anyons.

Let $P$ be a finitely generated free $R$-module equipped with an $\varepsilon$-hermitian form $\Omega$ as in~\eqref{eq:Omega_epsilon_hermitian}, and assume that $\Omega$ is unimodular; equivalently, the map
\begin{equation}
\lambda_P\colon P\longrightarrow P^*,
\qquad
q\longmapsto\Omega(\mathord\cdot,q),
\end{equation}
is an isomorphism. For a submodule $S\subset P$, write
\begin{equation}
S^\perp=\{p\in P: \Omega(p,s)=0\text{ for all }s\in S\}.
\end{equation}
Let $L\subset P$ be a Lagrangian submodule, meaning that $L=L^\perp$, and set $Q=P/L$. We have a~short exact sequence
\begin{equation}
0\longrightarrow L\xrightarrow{\iota}P\xrightarrow{\pi}Q\longrightarrow0
\label{eq:Lagrangian-short-exact-sequence}
\end{equation}

When $\varepsilon=-1$, this corresponds to a translation-invariant Pauli stabilizer code, referred to as \emph{exact} in~\cite{haah2013commuting}, \emph{topological} in~\cite{ellison2022pauli}, and \emph{Lagrangian} in~\cite{ruba2024homological,ruba2025witt}.

Since $\Omega$ vanishes on $L\times L$, it descends to nondegenerate pairings
\begin{align}
\Omega_{Q,L}\colon Q\times L&\longrightarrow R,
&
\Omega_{Q,L}([p],\ell)&=\Omega(p,\ell),
\label{eq:Lagrangian-pairing-QL}\\
\Omega_{L,Q}\colon L\times Q&\longrightarrow R,
&
\Omega_{L,Q}(\ell,[p])&=\Omega(\ell,p).
\label{eq:Lagrangian-pairing-LQ}
\end{align}
Hence $(Q,L,\Omega_{Q,L})$ and $(L,Q,\Omega_{L,Q})$ are duality pairs. Moreover, the induced map
\begin{equation}
\lambda_L\colon L\longrightarrow Q^*,
\qquad
\ell\longmapsto\Omega_{Q,L}(\mathord\cdot,\ell),
\end{equation}
is an isomorphism because $L^\perp=L$. Consequently, by Proposition~\ref{prop:eval},
\begin{equation}
\widetilde\Ext^0(\overline Q,R)=0,
\qquad
\widetilde\Tor_0(\widehat R,Q)=0.
\label{eq:Lagrangian-degree-zero-vanishing}
\end{equation}

Choose a resolution $G_\bullet\xrightarrow{\tau}L$ by finitely generated free modules. Splicing it with~\eqref{eq:Lagrangian-short-exact-sequence} gives a~resolution of $Q$,
\begin{equation}
\cdots\longrightarrow G_1\longrightarrow G_0
\xrightarrow{\iota\tau}P\xrightarrow{\pi}Q\longrightarrow0.
\end{equation}
For the duality pair $(Q,L,\Omega_{Q,L})$, the associated two-sided complex is
\begin{equation}
\mathbb F_{Q,L}\colon\quad
\cdots\longrightarrow G_1\longrightarrow G_0
\xrightarrow{\lambda_P\iota\tau}P^*
\xrightarrow{\tau^*\iota^*}G_0^*
\longrightarrow G_1^*\longrightarrow\cdots.
\label{eq:Lagrangian-two-sided-QL}
\end{equation}
For the reversed duality pair $(L,Q,\Omega_{L,Q})$, it is
\begin{equation}
\mathbb F_{L,Q}\colon\quad
\cdots\longrightarrow G_1\longrightarrow G_0
\xrightarrow{\iota\tau}P
\xrightarrow{\tau^*\iota^*\lambda_P}G_0^*
\longrightarrow G_1^*\longrightarrow\cdots.
\label{eq:Lagrangian-two-sided-LQ}
\end{equation}
After identifying $P$ with $P^*$ by $\lambda_P$, these are the same sequence with their gradings shifted by one. With the convention $(C[1])^j=C^{j+1}$, and after the standard sign adjustment in the differential of a shifted complex, this gives
\begin{equation}
\mathbb F_{L,Q}\cong\mathbb F_{Q,L}[1].
\label{eq:Lagrangian-two-sided-shift}
\end{equation}

\begin{prop}[Degree shift]
\label{prop:Lagrangian-degree-shift}
For every $p\geq0$, the connecting homomorphisms associated with~\eqref{eq:Lagrangian-short-exact-sequence} give canonical isomorphisms
\begin{equation}
\widetilde\Ext^p(\overline L,R)
\xrightarrow{\ \cong\ }
\widetilde\Ext^{p+1}(\overline Q,R),
\qquad
\widetilde\Tor_{p+1}(\widehat R,Q)
\xrightarrow{\ \cong\ }
\widetilde\Tor_p(\widehat R,L).
\label{eq:Lagrangian-degree-shifts}
\end{equation}
\end{prop}

\begin{proof}
The first isomorphism follows by taking cohomology in~\eqref{eq:Lagrangian-two-sided-shift}. After tensoring the same two-sided complexes with $\widehat R$, the corresponding homology groups give the second isomorphism. Under these identifications, the resulting maps agree with the usual connecting homomorphisms.
\end{proof}

\begin{cor}[Shifted complementary-degree duality]
\label{cor:Lagrangian-shifted-duality}
For every $0\leq p\leq d-1$, let
\begin{equation}
\operatorname{Br}_p\colon
\widetilde\Tor_p(\widehat R,Q)
\times
\widetilde\Tor_{d-p-1}(\widehat R,L)
\longrightarrow\widehat R
\end{equation}
be the pairing of Proposition~\ref{prop:cup_pairing}. Then
\begin{equation}
\operatorname{Br}^{L}_p(u,v)
:=
\operatorname{Br}_p\bigl(u,\partial_{d-p}v\bigr)
\label{eq:Lagrangian-shifted-pairing-definition}
\end{equation}
defines a canonical sesquilinear pairing 
\begin{equation}
\operatorname{Br}^{L}_p\colon
\widetilde\Tor_p(\widehat R,Q)
\times
\widetilde\Tor_{d-p}(\widehat R,Q)
\longrightarrow\widehat R.
\label{eq:Lagrangian-shifted-pairing}
\end{equation}

More explicitly, let $1\leq p\leq d-1$, and suppose that
\begin{equation}
u=[\alpha]\in\widetilde\Tor_p(\widehat R,Q),
\qquad
v=[\beta]\in\widetilde\Tor_{d-p}(\widehat R,Q),
\end{equation}
with
\begin{equation}
\alpha\in C^{d-p-1}(\Sigma;\mathcal Q),
\qquad
\beta\in C^{p-1}(\Sigma;\mathcal Q).
\end{equation}
For any lift $\widetilde\beta\in C^{p-1}(\Sigma;\mathcal P)$ of $\beta$, one has
\begin{align}
\operatorname{Br}^{L}_p(u,v)
&=
\int a^*\alpha\smile^{\Omega_{Q,L}}\delta\widetilde\beta \label{eq:Lagrangian-cup-and-homological-pairing} =
c_{d,p}\,\operatorname{ev}_p\Bigl(u,
\kappa_{d-p-1}\bigl(\partial_{d-p}v\bigr)\Bigr),
\end{align} 
where $\partial_{d-p}v\in  \widetilde\Tor_{d-p-1}(\widehat R,L)$ is the image of the second isomorphism in~\eqref{eq:Lagrangian-degree-shifts}.

Suppose, in addition, that every $\widetilde\Ext^q(\overline Q,R)$ is a finite length module. Then the pairings~\eqref{eq:Lagrangian-shifted-pairing} are perfect and induce canonical isomorphisms
\begin{equation}
\widetilde\Tor_{d-p}(\widehat R,Q)
\cong
\widetilde\Tor_p(\widehat R,Q)^\#,
\qquad
\widetilde\Ext^{d-p}(\overline Q,R)
\cong
\widetilde\Ext^p(\overline Q,R)^\#.
\label{eq:Lagrangian-complementary-isomorphisms}
\end{equation}
\end{cor}

\begin{proof}
Apply Propositions~\ref{prop:cup_pairing} and~\ref{prop:kappa_map} to the duality pair $(Q,L,\Omega_{Q,L})$, and then use the connecting isomorphism~\eqref{eq:Lagrangian-degree-shifts}. Formula~\eqref{eq:Lagrangian-cup-and-homological-pairing} follows from the cellular description of the connecting homomorphism and the factorization~\eqref{eq:Br_factorization}.
\end{proof}

\begin{remark}[Middle-degree self-pairing]
\label{rem:Lagrangian-middle-degree}
When $d$ is even,
\eqref{eq:Lagrangian-shifted-pairing} specializes to a canonical
self-pairing
\begin{equation}
\operatorname{Br}^{L}_{\frac d2}\colon
\widetilde\Tor_{\frac d2}(\widehat R,Q)
\times
\widetilde\Tor_{\frac d2}(\widehat R,Q)
\longrightarrow \widehat R.
\label{eq:Lagrangian-middle-degree-pairing}
\end{equation}
Its properties are discussed in Section \ref{sec:quadratic_Lagrangian}.
\end{remark}

\section{Quadratic forms} \label{sec:quadratic}

In the previous sections we have defined a flat resolution of $\widehat R$
using a cell structure of the sphere at infinity $S^{d-1}$. This allowed us to relate invariants $\widetilde{\operatorname{Tor}}$
of duality pairs to the cohomology of the complex $C^{\bullet}(\Sigma;\mathcal M)$, and to construct pairings between $\widetilde{\operatorname{Tor}}$
modules in terms of the cup product. 

We will now focus on special duality pairs discussed in Sections \ref{sec:self-pairing} and \ref{sec:Lagrangian}. In these cases,
for $d$ odd and $d$ even respectively, modules $\widetilde \Tor_{\frac{d-1}{2}}(\widehat R,P)$ and $\widetilde \Tor_{\frac{d}{2}}(\widehat R,Q)$
carry a self-pairing. This pairing is hermitian or anti-hermitian, depending on $d$ and the symmetry sign $\varepsilon \in \{ \pm 1 \}$
characterizing the sesquilinear form $\Omega$.  As we will establish below, a richer structure emerges: the~self-pairing of the relevant $\widetilde \Tor$ module is either alternating, or symmetric with a distinguished quadratic refinement. The construction of the quadratic refinements uses Steenrod's higher cup products $\smile_i$ and the antipodal involution on $S^{d-1}$.

\Needspace{5\baselineskip}
\subsection{Projective space and equivariant cohomology}
 
\subsubsection*{Group (co)homology}

If $M$ is an abelian group with an involution $a$ and $\chi \in \{ \pm 1 \}$, we write
\begin{equation}
M^\chi := \ker(1- \chi a), \qquad M_\chi := \operatorname{coker}(1- \chi a).
\end{equation}
In the case $\chi = 1$, these are the standard invariants and coinvariants of the group $C_2 = \{ 1 , a \}$ acting on $M$. If $\chi=-1$, we obtain invariants and coinvariants for the action twisted by the nontrivial character of $C_2$. Functors of invariants and coinvariants are exact from one side, and their derived functors go by the names of group cohomology and homology~\cite{Brown1982Cohomology}. They may be expressed as extension and torsion modules over the group ring $\mathbb Z [C_2]$: respectively $\Ext^\bullet_{\mathbb Z[C_2]}(\mathbb Z^\chi,M)$ and $\Tor_\bullet^{\mathbb Z[C_2]}(\mathbb Z^\chi, M)$, where $\mathbb Z^\chi$ is $\mathbb Z$ with the action of $a$ given by multiplication by $\chi$. In particular, $\mathbb Z^+$ is the trivial representation of $C_2$. Both homology and cohomology can be computed using the standard two-periodic $\mathbb Z[C_2]$-free resolution $W_\bullet(\chi) \longrightarrow \mathbb Z^\chi$, with $W_i(\chi) = \mathbb Z[C_2]$ for all $i \geq 0$:
\begin{equation}
\dots \longrightarrow W_2(\chi) \xrightarrow{D_2^\chi} W_1(\chi) \xrightarrow{D_1^\chi}  W_0(\chi)   \xrightarrow{\pi^\chi} \mathbb Z^\chi
\label{eq:2-periodic-resol}
\end{equation}
Up to sign, the differentials alternate between symmetrization and skew-symmetrization maps:
\begin{equation}
D_i^\chi : W_i(\chi) \longrightarrow W_{i-1}(\chi), \qquad   D_i^\chi = 1 + \chi (-1)^i a.
\end{equation}

\begin{remark}
    The resolution $W_\bullet(+1)$ may be viewed as the cellular chain complex of a $C_2$-equivariant cell structure on the infinite-dimensional sphere $ S^\infty$. This is a contractible space with a $C_2$-action, often denoted $E C_2$. The quotient of the $C_2$-action is the classifying space $BC_2 \cong \mathbb{RP}^\infty$. See \cite[Chapter~I, \S\S4 and~6]{Brown1982Cohomology} or
\cite[Example~1B.3]{Hatcher2002}.
\end{remark}

\subsubsection*{Equivariant homology of the projective space}

The quotient of $S^{d-1}$ by the antipodal $C_2$-action, introduced in Section~\ref{sec:antipodal}, is the real projective space $\mathbb{RP}^{d-1}$. The fan $\Sigma$ defines a triangulation of the sphere $S^{d-1}$ in which every simplex is disjoint from its antipode. Therefore, it induces a natural triangulation on the quotient space. Consequently, one may interpret $C_\bullet(\Sigma;\mathbb Z)_+$ and $C^\bullet(\Sigma;\mathbb Z)^+$ as complexes of simplicial chains and cochains of $\mathbb{RP}^{d-1}$, respectively. 

The complexes $C_\bullet(\Sigma;\mathbb Z)_-$ and $C^\bullet(\Sigma;\mathbb Z)^-$ admit a similar interpretation using chains and cochains with coefficients in a local system. Namely, let $\mathbb Z^-$ be the local system on $\mathbb{RP}^{d-1}$ associated with the double cover $S^{d-1}\longrightarrow \mathbb{RP}^{d-1}$ and the sign representation of $C_2$. Then $C_\bullet(\Sigma;\mathbb Z)_-$ and $C^\bullet(\Sigma;\mathbb Z)^-$ are the corresponding simplicial chain and cochain complexes. Writing $\mathbb Z^+=\mathbb Z$ for the constant coefficient system, we may use the notation $\mathbb Z^\chi$, $\chi\in \{\pm 1 \}$, for both cases. We~denote the homology of $C_\bullet(\Sigma;\mathbb Z)_\chi$ by $
H_\bullet(\mathbb{RP}^{d-1};\mathbb Z^\chi)$.

Keeping these beautiful geometric interpretations in mind, we~will work mostly algebraically.

One of the objects of interest will be the complexes $C^\bullet(\Sigma;\mathcal R)^\pm$. More precisely, we will encounter cocycles which are not $\pm$-invariant on the nose, but rather only up to certain exact terms. To deal with this, we must replace strict invariants with homotopy invariants. We do this by replacing $C^\bullet(\Sigma;\mathcal R)^\pm$ with a quasi-isomorphic total complex that combines $C^\bullet(\Sigma;\mathcal R)$ with the resolution $W_\bullet(\chi)$, see \eqref{eq:2-periodic-resol}. Before we delve into that, let us present the corresponding homotopy-coinvariant construction for chains with $\mathbb Z$ coefficients.

Taking the tensor product $W_\bullet(\chi) \otimes_{\mathbb Z[C_2]} C_\bullet(\Sigma;\mathbb Z)$, with the $C_2$ action on $C_\bullet(\Sigma;\mathbb Z)$ given by the pushforward maps $a_*$, and using the canonical isomorphisms $W_i(\chi) \otimes_{\mathbb Z[C_2]} C_p(\Sigma;\mathbb Z) \cong C_p(\Sigma;\mathbb Z)$, we obtain the double complex
\begin{equation}
\begin{tikzcd}[column sep=small, row sep=large]
\cdots  &
\vdots \arrow[d,"\partial"]  &
\vdots \arrow[d,"\partial"]  &
\vdots \arrow[d,"\partial"] \\
\cdots \arrow[r] &
C_p(\Sigma;\mathbb Z)
\arrow[r,"D_2^\chi"] \arrow[d,"\partial"] &
C_p(\Sigma;\mathbb Z)
\arrow[r,"D_1^\chi"] \arrow[d,"\partial"] &
C_p(\Sigma;\mathbb Z)
\arrow[d,"\partial"] \\
\cdots \arrow[r] &
C_{p-1}(\Sigma;\mathbb Z)
\arrow[r,"D_2^\chi"] \arrow[d,"\partial"] &
C_{p-1}(\Sigma;\mathbb Z)
\arrow[r,"D_1^\chi"] \arrow[d,"\partial"] &
C_{p-1}(\Sigma;\mathbb Z)
\arrow[d,"\partial"] \\
&
\vdots &
\vdots &
\vdots 
\end{tikzcd}
\label{eq:equivariant_simplicial_chain_double}
\end{equation}
where $D_i^\chi = 1 + \chi (-1)^i a_*$. Rows of this double complex can be augmented to the right by the quotient map $\pi^\chi \colon C_\bullet(\Sigma;\mathbb Z) \to C_\bullet(\Sigma;\mathbb Z)_\chi$ to $\chi$-twisted coinvariants.

Consider the total complex of the double complex above:
\begin{equation}
    C_\bullet^{C_2}(\Sigma;\mathbb Z^\chi) = \Tot(W_\bullet(\chi) \otimes_{\mathbb Z[C_2]} C_\bullet(\Sigma;\mathbb Z)).
\end{equation}
Then we have
\begin{equation}
    C_p^{C_2}(\Sigma;\mathbb Z^\chi) \cong \bigoplus_{i=0}^p C_{p-i}(\Sigma;\mathbb Z).  
\end{equation}
Choosing the sign conventions of \eqref{eq:total_differential} in Appendix \ref{app:double_complex}, we have
\begin{equation}
    \partial_{\Tot} = \partial + (-1)^p D_{i}^\chi \qquad \text{on} \qquad W_i(\chi) \otimes_{\mathbb Z[C_2]} C_p(\Sigma;\mathbb Z).
\end{equation}
The complex $C_\bullet^{C_2}(\Sigma;\mathbb Z^\chi)$ is a simplicial version of the Borel construction of $\chi$-twisted $C_2$-equivariant homology of $S^{d-1}$. Since the $C_2$-action is free, equivariant homology reduces to homology of the quotient space with local coefficients.

\begin{lem}
The homology $H_\bullet^{C_2}(\Sigma;\mathbb Z^\chi)$ of the complex $C_\bullet^{C_2}(\Sigma;\mathbb Z^\chi)$ can be canonically identified with the homology of the projective space:
\begin{equation}
H_p^{C_2}(\Sigma;\mathbb Z^\chi) \cong 
H_p(\mathbb{RP}^{d-1},\mathbb Z^\chi),
\label{eq:equivariant_homology_iso}
\end{equation}
with integer coefficients if $\chi=1$ and with local coefficients $\mathbb Z^{-}$ if $\chi=-1$. An explicit quasi-isomorphism of the underlying complexes is given by the augmentation:
\begin{align}
    C_p^{C_2}(\Sigma;\mathbb Z^\chi) & \longrightarrow C_p(\Sigma;\mathbb Z)_\chi, \label{eq:equivariant_double_complex_aug}  \\
    \bigoplus_{i=0}^p C_{p-i}(\Sigma;\mathbb Z) \ni (a_0,\dots,a_p) & \longmapsto \pi^\chi (a_0) \in C_p(\Sigma;\mathbb Z)_\chi. \nonumber
\end{align}
\end{lem}
\begin{proof}
The map \eqref{eq:equivariant_double_complex_aug} is, up to a change of grading and switching the roles of rows and columns in a double complex, the map $\partial_A$ discussed in Lemma~\ref{lem:edge_A}. We only have to check that augmented rows of \eqref{eq:equivariant_simplicial_chain_double} are exact. Indeed, the $C_2$-action on $\Sigma$ is free, so $C_n(\Sigma;\mathbb Z)$ are free $\mathbb Z[C_2]$-modules. Therefore, $\ker(1\pm a_*) = \operatorname{im}(1 \mp a_*)$. Moreover, $\ker(\pi^\chi) = \operatorname{im}(D_1^\chi)$ by definition.
\end{proof}

Let us suppose momentarily that $d$ is even. Then the antipodal map is orientation-preserving, and $S^{d-1}$ induces an orientation of the projective space. The fundamental class of $\mathbb{RP}^{d-1}$, which is represented by a uniquely determined simplicial cycle $[\mathbb{RP}^{d-1}] \in C_{d-1}(\Sigma;\mathbb Z)_+$, generates the homology
\begin{equation}
H_{d-1}(\mathbb{RP}^{d-1},\mathbb Z) \cong \mathbb Z.
\end{equation}
It satisfies
\begin{equation}
    \pi^{+} [S^{d-1}] = 2 [\mathbb{RP}^{d-1}],
    \label{eq:twofold_cover}
\end{equation}
due to the two-fold covering of the projective space by $S^{d-1}$. We define 
\begin{equation}
 \Gamma \in C_{d-1}^{C_2}(\Sigma;\mathbb Z^+)   
\end{equation}
to be a cycle which corresponds to the fundamental class via the isomorphism \eqref{eq:equivariant_homology_iso} with $p=d-1$ and $\chi=+1$. It is unique up to $\partial_\Tot$-boundaries.

Next, let us consider the case of odd $d$. If $d \neq 1$, the antipodal map is orientation-reversing, and $\mathbb{RP}^{d-1}$ is non-orientable. Its top degree integral homology vanishes, and instead we will consider homology with coefficient system $\mathbb Z^-$. For $d=1$, the relevant projective space is a point and the preceding remarks do not apply. Moreover, local systems $\mathbb Z^+$ and $\mathbb Z^-$ are isomorphic for $d=1$. In either case, we have
\begin{equation}
H_{d-1}(\mathbb{RP}^{d-1},\mathbb Z^{-}) \cong \mathbb Z, \qquad d \text{ odd}.
\end{equation}
We let $[\mathbb{RP}^{d-1}] \in C_{d-1}(\Sigma;\mathbb Z)_-$ be the cycle uniquely characterized by 
\begin{equation}
 \pi^{-} [S^{d-1}]=2[\mathbb{RP}^{d-1}].   
 \label{eq:twofold_cover_minus}
\end{equation}
For brevity, we will refer to it as the fundamental class, just as in the case of even $d$. Furthermore, we define $ \Gamma \in C_{d-1}^{C_2}(\Sigma;\mathbb Z^-)  $ to be a cycle corresponding to $[\mathbb{RP}^{d-1}]$ via the isomorphism \eqref{eq:equivariant_homology_iso} with $p=d-1$ and $\chi=-1$. It is unique up to $\partial_\Tot$-boundaries.

Let us summarize the unified setup for both even and odd $d$, and establish the explicit formulas required for forthcoming computations. The equivariant chain $\Gamma \in C_{d-1}^{C_2}(\Sigma;\mathbb Z^{(-1)^d})$ is a tuple $\Gamma=(\gamma_0,\gamma_1,\dots,\gamma_{d-1})$ with $\gamma_i\in C_{d-i-1}(\Sigma;\mathbb Z)$.
The cycle condition $\partial_{\Tot} \Gamma = 0$ is equivalent to the following homological descent equations:
\begin{equation}
\partial \gamma_i + \big( (-1)^{d+i} -a_* \big) \gamma_{i+1}=0, \qquad 0 \leq i \leq d-2.
\label{eq:gamma_homological_descent}
\end{equation}
Because the augmented rows of the double complex \eqref{eq:equivariant_simplicial_chain_double} are exact, one can construct $\Gamma$ inductively. Once $\gamma_i$ is specified, the exactness guarantees the existence of $\gamma_{i+1}$ satisfying \eqref{eq:gamma_homological_descent}. Any two cycles $\Gamma$ obtained this way are homologous. The base of this inductive construction is the leading term $\gamma_0$, which must lift the fundamental class $[\mathbb{RP}^{d-1}]$. We can construct $\gamma_0$ explicitly by choosing a fundamental domain for the $C_2$-action. Specifically, let $\Sigma_{d-1}^\uparrow \subset \Sigma_{d-1}$ be a subset such that $\Sigma_{d-1}$ is the disjoint union of $\Sigma_{d-1}^\uparrow$ and its involution: $\Sigma_{d-1} = \Sigma_{d-1}^\uparrow \sqcup (- \Sigma_{d-1}^\uparrow)$. For instance, one can take $\Sigma_{d-1}^\uparrow$ to be the top simplices of a triangulation of a hemisphere of $S^{d-1}$. Recall the formula \eqref{eq:sphere_fund_class} for the fundamental class of the sphere, and note that with our choices of orientations
\begin{equation}
    [S^{d-1}:-\sigma] = (-1)^d [S^{d-1}:\sigma].
    \label{eq:fund_class_symmetry}
\end{equation}
We set
\begin{equation}
\gamma_0 =  \sum_{\sigma \in \Sigma_{d-1}^\uparrow} [S^{d-1}:\sigma] \, \sigma   .     
\end{equation}
Then $(1+(-1)^d a_*)\gamma_0 = [S^{d-1}]$ by \eqref{eq:fund_class_symmetry}, and hence $\pi^{(-1)^d}(\gamma_0) = [\mathbb{RP}^{d-1}]$.

\subsubsection*{Homotopy-invariant $\mathcal R$-valued cochains}

We now turn to consider the double complex \begin{equation} \Hom_{\mathbb Z[C_2]}( W_\bullet(\chi) ,C^\bullet(\Sigma;\mathcal R)),\end{equation} with the $C_2$ action on $C^\bullet(\Sigma;\mathcal R)$ given by the pullback $a^*$. Identifying 
\begin{equation}
\Hom_{\mathbb Z[C_2]}( W_i(\chi) ,C^p(\Sigma;\mathcal R)) \cong C^p(\Sigma;\mathcal R),
\end{equation}
we~find 
\begin{equation}
\begin{tikzcd}[column sep=small, row sep=large]
\vdots \arrow[d,"\delta"] &
\vdots \arrow[d,"\delta"] &
\vdots \arrow[d,"\delta"] &
\\
C^p(\Sigma;\mathcal R)
\arrow[r,"D^0_\chi"] \arrow[d,"\delta"] &
C^p(\Sigma;\mathcal R)
\arrow[r,"D^1_\chi"] \arrow[d,"\delta"] &
C^p(\Sigma;\mathcal R)
\arrow[r,"D^2_\chi"] \arrow[d,"\delta"] &
\cdots \\
C^{p+1}(\Sigma;\mathcal R)
\arrow[r,"D^0_\chi"] \arrow[d,"\delta"] &
C^{p+1}(\Sigma;\mathcal R)
\arrow[r,"D^1_\chi"] \arrow[d,"\delta"] &
C^{p+1}(\Sigma;\mathcal R)
\arrow[r,"D^2_\chi"] \arrow[d,"\delta"] &
\cdots \\
\vdots &
\vdots &
\vdots &
\end{tikzcd}
\end{equation}
where $D^i_\chi = (1-\chi (-1)^i a^*)$. Rows of this double complex can be augmented from the left by the inclusion map $\iota^\chi \colon C^\bullet(\Sigma;\mathcal R)^\chi \longrightarrow C^\bullet(\Sigma;\mathcal R)$ of $\chi$-twisted invariants. 

Consider the total complex of the double complex above:
\begin{equation}
    C^\bullet_{C_2}(\Sigma;\mathcal R^\chi) = \Tot \Big(\Hom_{\mathbb Z[C_2]} \big(W_\bullet(\chi),C^\bullet(\Sigma;\mathcal R) \big) \Big).
    \label{eq:equivariant_cochain_decomp}
\end{equation}
Then we have
\begin{equation}
    C^p_{C_2}(\Sigma;\mathcal R^\chi) \cong \bigoplus_{i=0}^p C^{p-i}(\Sigma;\mathcal R).  
\end{equation}
We choose sign conventions for the total differential $\delta_{\Tot}$ as defined in \eqref{eq:total_differential} in Appendix \ref{app:double_complex}. That is,
\begin{equation}
    \delta_{\Tot} = \delta + (-1)^p D^{i}_\chi \qquad \text{on} \qquad \Hom_{\mathbb Z[C_2]} \big(W_i(\chi) , C^p(\Sigma;\mathcal R) \big).
\end{equation}
More explicitly, a cochain $ \Xi \in C^p_{C_2}(\Sigma;\mathcal R^\chi)$ is a tuple
\begin{equation}
\Xi=(\xi_0,\xi_1,\ldots,\xi_p),
\qquad
\xi_i\in C^{p-i}(\Sigma;\mathcal R).
\end{equation}
The total differential is given by
\begin{equation}
\delta_{\Tot}\Xi 
=
\bigl(
\delta \xi_0, \delta \xi_1 + (-1)^{p} D^{0}_\chi \xi_0, \delta \xi_2 - (-1)^p D^1_\chi \xi_1, \dots \bigr) .
\end{equation}
Consequently, $\Xi$ is a cocycle if its components $\xi_i $ satisfy the following descent equations:
\begin{align}
\delta \xi_0 &=0 \\
    \delta \xi_i + (-1)^p ((-1)^{i-1} - \chi a^*) \xi_{i-1} &= 0, \qquad 1 \leq i \leq p ,  \nonumber \\
    (1 - \chi (-1)^p a^*) \xi_p & = 0. \nonumber
\end{align}

\subsubsection*{Pairings of chains and cochains} There exists a natural pairing
\begin{equation}
    \langle - , - \rangle  \colon C_\bullet^{C_2}(\Sigma;\mathbb Z^\chi) \times C^\bullet_{C_2} (\Sigma;\mathcal R^\chi) \longrightarrow \widehat R,
\end{equation}
which, for $\Upsilon \in C_p^{C_2}(\Sigma;\mathbb Z^\chi)$ and $\Xi  \in C^p_{C_2}(\Sigma;\mathcal R^\chi)$, is given by
\begin{equation}
\langle \Upsilon , \Xi \rangle = \sum_{i=0}^p \langle \upsilon_i, \xi_i \rangle.
\end{equation}
It satisfies
\begin{equation}
    \langle \partial_{\Tot} \Upsilon , \Xi \rangle_0 = \langle \Upsilon, \delta_{\Tot} \Xi \rangle_0.
    \label{eq:tot_pairing}
\end{equation}
We will consider this pairing mostly with $p = d-1$, $\Upsilon = \Gamma$ and suitably chosen cocycle $\Xi$. Then \eqref{eq:tot_pairing} implies that $\langle \Gamma,\Xi \rangle_0$ does not depend on the arbitrary choices involved in the definition of~$\Gamma$, and is invariant under cohomological changes of $\Xi$. 

\subsection{Higher cup products}

In this section we define Steenrod's higher cup products $\smile_i$ generalized to $\mathcal R$-valued cochains. We also introduce products $\smile_i^\Omega$ generalizing the $\smile^\Omega$ introduced in Definition \ref{def:smile_Omega_product}, and we use them to derive symmetry properties of pairings $\operatorname{Br}$ given by the cup product.

Higher cup products have been used by Steenrod~\cite{steenrod1947products} to construct celebrated squaring operations in cohomology with $\mathbb F_2$ coefficients, and here they play an essential role in our construction of quadratic forms. They measure the failure of the cup product $\smile$ to be graded-commutative at the cochain level. In our setting, this failure of graded commutativity is reflected by the lack of symmetry or skew-symmetry under $a^*$ of cochains of the form $a^* \alpha \smile^\Omega \alpha$ or $a^* \alpha \smile^\Omega \delta \alpha$. We~use higher cup products to explicitly exhibit symmetry or skew-symmetry up to homotopy terms. 

\begin{defn}[Higher cup products] \label{def:higher-cup}
Let $i\geq 0$. The $i$-th cup product of $f\in C^p(\Sigma;\mathcal R)$ and
$g\in C^q(\Sigma;\mathcal R)$ is the cochain
\begin{equation}
f\smile_i g
\in
C^{p+q-i}(\Sigma;\mathcal R),
\end{equation}
defined as follows. If $i > \min \{ p , q \}$, we set $f\smile_i g=0$. Otherwise, we put $n=p+q-i$ and let $\sigma\in\Sigma_n$ be the cone spanned by the ordered rays
\begin{equation}
\rho_0<\rho_1<\cdots<\rho_n.
\end{equation}
We will give a formula for the $\sigma$-component of $f\smile_i g$.

For a subset $A\subseteq\{0,\ldots,n\}$, denote by $\sigma_A$ the face
of $\sigma$ spanned by the rays $\rho_a$ with $a\in A$.

An overlapping partition of $\{0,\ldots,n\}$ into $i+2$ intervals is
specified by integers $\ell_1,\dots,\ell_{i+1}$ satisfying
\(
0\leq \ell_1<\ell_2<\cdots<\ell_{i+1}\leq n.
\)
Set $\ell_0=0$, $\ell_{i+2}=n$, and
\begin{equation}
L_r
=
\{\ell_r,\ell_r+1,\ldots,\ell_{r+1}\},
\qquad
0\leq r\leq i+1.
\end{equation}
Associated to the overlapping partition
$\mathcal L=(L_0,\ldots,L_{i+1})$, define
\begin{equation}
A_{\mathcal L}
=
\bigcup_{r \text{ even}}
L_r,
\qquad
B_{\mathcal L}
=
\bigcup_{r \text{ odd}}
L_r.
\end{equation}
Thus
\begin{equation}
A_{\mathcal L}\cup B_{\mathcal L}
=
\{0,\ldots,n\},
\qquad
A_{\mathcal L}\cap B_{\mathcal L}
=
\{\ell_1,\ldots,\ell_{i+1}\}.
\end{equation}

Let $w_{\mathcal L}$ be the permutation of $\{0,\ldots,n\}$ obtained
by first listing the elements of $A_{\mathcal L}$ in increasing order
and then listing the elements of
$B_{\mathcal L}\setminus A_{\mathcal L}$, also in increasing order. Let $\operatorname{sgn}(w_{\mathcal L}) \in \{ \pm 1 \}$ be the sign of $w_{\mathcal L}$.

The $\sigma$-component of the higher cup product is defined to be
\begin{equation}
(f\smile_i g)_\sigma
=
\sum_{\substack{\mathcal L \text{ such that:}\\
                 |A_{\mathcal L}|=p+1\\
                 |B_{\mathcal L}|=q+1}}
\operatorname{sgn}(w_{\mathcal L})
f_{\sigma_{A_{\mathcal L}}}
\cdot
g_{\sigma_{B_{\mathcal L}}}
\in
\mathcal R(\sigma).
\end{equation}
As in Definition~\ref{def:cup}, the two factors in each summand are
implicitly mapped to $\mathcal R(\sigma)$ by the corresponding face
inclusions. 
\end{defn}

For $i=0$, the only overlapping partition that contributes has
$\ell_1=p$, and Definition~\ref{def:higher-cup} reduces to
Definition~\ref{def:cup}. Thus
\begin{equation}
f\smile_0 g=f\smile g.
\end{equation}
We also adopt the convention $f\smile_{-1}g=0$. We also remark that multiple sign conventions for higher cup products occur in the literature. Definition \ref{def:higher-cup} is a direct adaptation of the definition in \cite{steenrod1947products} to our setting.

\begin{prop}\label{prop:cup_differential}
For $f\in C^p(\Sigma;\mathcal R)$ and
$g\in C^q(\Sigma;\mathcal R)$, the higher cup products satisfy
\begin{equation}
\begin{split}
&\delta(f\smile_i g)
-(\delta f)\smile_i g
-(-1)^p f\smile_i(\delta g)\\
&\qquad=
(-1)^{p+q+i}f\smile_{i-1}g
+
(-1)^{p+q+pq}g\smile_{i-1}f.
\end{split}
\label{eq:Steenrod_relation}
\end{equation}
Moreover, higher cup products are compatible with $a^*$:
\begin{equation}
    a^* (f \smile_i g) = (a^* f) \smile_i (a^* g).
    \label{eq:a_naturality_cup_i}
\end{equation}
\end{prop}
The proof of \eqref{eq:Steenrod_relation} is the same combinatorial calculation as in Section~5 of~\cite{steenrod1947products}. The proof of \eqref{eq:a_naturality_cup_i} is analogous to that of Lemma \ref{lem:a-naturality-cup}.

In particular, if $f$ and $g$ are cocycles, then
\begin{equation}
\delta(f\smile_1 g)
=
(-1)^{p+q+1}(f\smile g
-
(-1)^{pq}g\smile f),
\end{equation}
so $\smile_1$ is an explicit cochain homotopy witnessing the graded
commutativity of the ordinary cup product.

\begin{defn} \label{def:higher_smile_Omega_product}
Let $(M,N,\Omega)$ be a duality pair. We define the sesquilinear pairing
\begin{equation}
 C^p(\Sigma; \mathcal M) \times C^q(\Sigma;\mathcal N) \longrightarrow C^{p+q-i} (\Sigma;\mathcal R), \qquad (\alpha,\beta) \mapsto a^* \alpha \smile_i^\Omega \beta,
 \end{equation} 
by first specifying it on simple tensors. If $\alpha = f \otimes m$ and $\beta = g \otimes n$, with
\begin{equation} 
f \in C^p(\Sigma;\mathcal R), \qquad g \in C^q(\Sigma;\mathcal R), \qquad m \in M, \qquad n \in N,
\end{equation}
we set:
\begin{equation}
    a^* \alpha \smile^\Omega_i \beta = \Omega(m,n) \, a^* f \smile_i g.
\end{equation}
\end{defn}

\begin{prop} \label{prop:P_val_Steenrod}
Let $P$ be a module with an $\varepsilon$-hermitian self-pairing $\Omega$, see \eqref{eq:Omega_epsilon_hermitian}. Then for $\alpha \in C^p(\Sigma;\mathcal P)$ and
$\beta \in C^q(\Sigma;\mathcal P)$, we have
\begin{align}
    &\delta( a^* \alpha \smile_i^\Omega \beta ) - a^* \delta \alpha \smile_i^\Omega \beta -(-1)^p a^* \alpha \smile_i^\Omega \delta \beta \label{eq:P_valued_Steenrod} \\
    & = (-1)^{p+q+i} a^* \alpha \smile_{i-1}^\Omega \beta + \varepsilon (-1)^{p+q+pq} a^*( a^* \beta \smile_{i-1}^\Omega \alpha ). \nonumber
\end{align}
\end{prop}
\begin{proof}
    By sesquilinearity, it suffices to verify \eqref{eq:P_valued_Steenrod} for simple tensors. This reduces to Proposition \ref{prop:cup_differential} and $\varepsilon$-hermiticity of $\Omega$. 
\end{proof}

\begin{cor} \label{cor:cup_symmetry}
Let $P, \alpha, \beta$ be as in Proposition \ref{prop:P_val_Steenrod}. If $\delta \alpha = \delta \beta=0$, then
\begin{equation}
    \int a^* \alpha \smile^\Omega \beta = \varepsilon  (-1)^{pq} (-1)^d \overline{\int a^* \beta \smile^\Omega \alpha}.
\end{equation}
\end{cor}
\begin{proof}
    Relation \eqref{eq:P_valued_Steenrod} shows that $a^* \alpha \smile^\Omega \beta$ equals $\varepsilon (-1)^{pq} a^*(a^* \beta \smile^\Omega \alpha)$ up to a coboundary term. Therefore, using \eqref{eq:push_pull_conjugate},
    \begin{align}
        \langle [S^{d-1}] , a^* \alpha \smile^\Omega \beta \rangle &= \varepsilon (-1)^{pq}  \overline{\langle  a_* [S^{d-1}], a^* \beta \smile^\Omega \alpha \rangle} \\
        &= \varepsilon (-1)^{pq} (-1)^d \overline{\langle  [S^{d-1}], a^* \beta \smile^\Omega \alpha \rangle}   . \nonumber
    \end{align}
\end{proof}

\subsection{Construction for modules with self-pairings} \label{sec:quadratic_self_pair}
Let $P$ be a module with an $\varepsilon$-hermitian self-pairing as in Section \ref{sec:self-pairing}. We~will assume that $d$ is odd throughout. Then the form $\operatorname{Br}_p$ in the middle degree $p = \frac{d-1}{2}$ becomes a self-pairing of the module $A = \widetilde \Tor_{\frac{d-1}{2}}(\widehat R,P)$. We study the $k$-bilinear form $b$ given by the scalar part of $\operatorname{Br}_{\frac{d-1}{2}}$. Depending on the sign $\chi = \varepsilon (-1)^{\frac{d-1}{2}}$, this bilinear form is either alternating or symmetric with a~distinguished quadratic refinement~$q$. 

\begin{defn} \label{def:self_pairing_xi}
For $\alpha \in C^{\frac{d-1}{2}}(\Sigma;\mathcal P)$ we set
\begin{equation}
 \Xi(\alpha) = (\xi_0(\alpha),\xi_1(\alpha),\dots), \qquad \xi_i(\alpha) = a^* \alpha \smile_i^\Omega \alpha.   
\end{equation}
\end{defn}

\begin{lem} \label{lem:self-pairing-Xi}
    Let $\alpha, \beta \in C^{\frac{d-1}{2}}(\Sigma;\mathcal P)$ and set $\chi = \varepsilon (-1)^{\frac{d-1}{2}}$. Viewing $\Xi(\alpha), \Xi(\beta)$ as cochains in $C^{d-1}_{C_2}(\Sigma;\mathcal R^\chi)$, the following statements hold.
    \begin{enumerate}
        \item If $\alpha$ is a $\delta$-cocycle, then $\Xi(\alpha)$ is a $\delta_\Tot$-cocycle. If~$\alpha$ is a $\delta$-coboundary, then $\Xi(\alpha)$ is a~$\delta_\Tot$-coboundary.
        \item
        If $\alpha$ is a $\delta$-cocycle and $\beta$ is a $\delta$-coboundary, then $\Xi(\alpha+\beta)-\Xi(\alpha)$ is a $\delta_\Tot$-coboundary. 
        \item If $\alpha$ is a $\delta$-cocycle, then
        \begin{equation}
            \int a^* \alpha \smile^\Omega \alpha = \langle \Gamma, \Xi( \alpha) \rangle  - \chi \overline{\langle \Gamma,\Xi(\alpha) \rangle}.
        \label{eq:self-pairing-divide-by-two}
        \end{equation}
        \item If $\chi =-1$ and $\alpha,\beta$ are $\delta$-cocycles, then
        \begin{equation}
        \left(    \int a^* \alpha \smile^\Omega \beta \right)_0 = \langle \Gamma , \Xi(\alpha+\beta) - \Xi(\alpha)-\Xi(\beta) \rangle_0.
        \label{eq:xi_refinement_identity_self_pair}
        \end{equation}
    \end{enumerate}
\end{lem}
\begin{proof}
(1) Consider \eqref{eq:P_valued_Steenrod} for $p =q= \frac{d-1}{2}$ and $\alpha=\beta$. Assuming that $\alpha$ is a cocycle, the identity reduces to
\begin{equation}
    \delta \left( a^* \alpha \smile_i^\Omega \alpha \right) = ((-1)^i + \varepsilon (-1)^p a^*) (a^* \alpha \smile_{i-1}^\Omega \alpha).
\end{equation}
Setting $\chi = \varepsilon (-1)^{\frac{d-1}{2}}$, this becomes
\begin{equation}
    \delta \xi_i(\alpha) + (-1)^{i-1} D^{i-1}_\chi \xi_{i-1}(\alpha)=0,
\end{equation}
or equivalently $\delta_\Tot \Xi (\alpha) = 0$. 

Next, suppose that $\alpha= \delta \theta$. Set $\Xi'(\theta) = (\xi_0'(\theta), \xi_1'(\theta),\dots)$, where
\begin{equation}
 \xi_i'(\theta) = a^* \theta \smile_i^\Omega \delta \theta + \varepsilon a^*( a^* \theta \smile_{i-1}^\Omega \theta).
\end{equation}
Applying \eqref{eq:P_valued_Steenrod} repeatedly one verifies that
\(
    \Xi(\delta \theta) = \delta_\Tot \Xi'(\theta).
\)

(2) By (1), it is enough to show that $\Xi(\alpha+\beta)-\Xi(\alpha)-\Xi(\beta)$ is a coboundary. We have
\begin{equation}
    \xi_i(\alpha+\beta) - \xi_i(\alpha)-\xi_i(\beta) = a^* \alpha \smile_i^\Omega \beta + a^* \beta \smile_i^\Omega \alpha.
    \label{eq:Xi_lineator}
\end{equation}
Choose $\theta$ such that $\beta = \delta \theta$. One verifies using \eqref{eq:P_valued_Steenrod} that the right hand side of \eqref{eq:Xi_lineator} equals $\delta_\Tot \Phi$, where $\Phi = (\phi_0, \phi_1,\dots)$ is given by
\begin{equation}
    \phi_i = a^* \theta \smile_i^\Omega \alpha + (-1)^{\frac{d-1}{2}} a^* \alpha \smile_i^\Omega \theta.
\end{equation}

(3) The cocycle equation satisfied by $\Xi( \alpha)$ is equivalent to equations
\begin{equation}
\chi a^* \xi_i(\alpha) = (-1)^i \xi_i(\alpha) + \delta \xi_{i+1}(\alpha),
\label{eq:cohomological_descent_self_pairing_case}
\end{equation}
where we adopt the convention that $\xi_i(\alpha) =0$ for $i<0$ and for $i> \frac{d-1}{2}$. Now we observe the following chain of equalities:
\begin{align}
    \langle [S^{d-1}], a^* \alpha \smile^\Omega \alpha \rangle &= \langle (1-a_*) \gamma_0, \xi_0(\alpha) \rangle = \langle \gamma_0 , \xi_0(\alpha) \rangle - \overline{\langle \gamma_0 , a^* \xi_0(\alpha) \rangle} \nonumber \\
    & = \langle \gamma_0 , \xi_0(\alpha) \rangle - \chi \overline{\langle \gamma_0, \xi_0(\alpha) \rangle} - \chi \overline{\langle \partial \gamma_0 , \xi_1(\alpha) \rangle} ,
\end{align}
where in the passage to the second line we used \eqref{eq:cohomological_descent_self_pairing_case} for $i=0$ and the adjointness of $\delta$ and $\partial$. In subsequent steps, we express $\partial \gamma_i$ in terms of $\gamma_{i+1}$ using \eqref{eq:gamma_homological_descent} and $a^* \xi_i(\alpha)$ in terms of $\xi_i(\alpha),\delta \xi_{i+1}(\alpha)$ using \eqref{eq:cohomological_descent_self_pairing_case}:
\begin{align}
\langle \partial \gamma_i , \xi_{i+1} (\alpha)\rangle &= (-1)^i \langle \gamma_{i+1}, \xi_{i+1}(\alpha) \rangle + \overline{\langle \gamma_{i+1}, a^* \xi_{i+1} (\alpha) \rangle } \\
& = (-1)^i \left( \langle \gamma_{i+1}, \xi_{i+1}(\alpha) \rangle - \chi \overline{\langle \gamma_{i+1}, \xi_{i+1}(\alpha) \rangle} \right)  \nonumber  \\
& +  \chi \overline{\langle \partial \gamma_{i+1}, \xi_{i+2}(\alpha) \rangle}. \nonumber
\end{align}
Iterating this step, which is easily formalized as an induction argument, we obtain \eqref{eq:self-pairing-divide-by-two}.

(4) The proof follows similar steps as for (3).
\end{proof}

\begin{proof}[Proof of Theorem \ref{thm:intro-quadratic} (a)]
    Recall the description of $A=\widetilde \Tor_{\frac{d-1}{2}}(\widehat R,P)$ as a subquotient of the cellular cochain module $C^{\frac{d-1}{2}}(\Sigma;\mathcal P)$, as discussed in the proof of Proposition \ref{prop:cup_pairing} (including the special case $d=1$). Recall also that $\operatorname{Br}_{\frac{d-1}{2}}$ was defined as the map induced on this subquotient by the cup product pairing (with $p = \frac{d-1}{2}$) introduced in Definition \ref{def:smile_Omega_pairing}. It follows immediately that the bilinear form $b$, defined as the scalar part of $\operatorname{Br}_{\frac{d-1}{2}}$, is induced in the same way from the $k$-bilinear form on $C^{\frac{d-1}{2}}(\Sigma;\mathcal P)$ given by
    \begin{equation}
        (\alpha,\beta) \mapsto \left( \int a^* \alpha \smile^\Omega \beta \right)_0.
    \end{equation}
   Since we are interested in the pairing induced on cohomology, it is enough to consider cocycles $\alpha,\beta$. Then \eqref{eq:self-pairing-divide-by-two} gives
    \begin{equation}
         \left( \int a^* \alpha \smile^\Omega \alpha \right)_0 = (1- \chi) \langle \Gamma, \Xi(\alpha) \rangle_0.
    \end{equation}
    If $\chi=1$, this readily implies that $b$ is alternating.

    Now suppose that $\chi = -1$. The symmetry of $b$ is a special case of Corollary \ref{cor:cup_symmetry}. Alternatively, it can be deduced from \eqref{eq:xi_refinement_identity_self_pair}. We will construct a quadratic refinement of $b$. Lemma \ref{lem:self-pairing-Xi} (a), (b) combined with the identity \eqref{eq:tot_pairing} imply that the map 
    \begin{equation}
     \alpha \longmapsto \langle \Gamma, \Xi(\alpha) \rangle_0 \in k
     \label{eq:Gamma_Xi_0}
    \end{equation}
    induces a function $H^{\frac{d-1}{2}}(\Sigma;\mathcal P) \longrightarrow k$. We have $H^{\frac{d-1}{2}}(\Sigma;\mathcal P) \cong \widetilde \Tor_{\frac{d-1}{2}}(\widehat R,P)$ if $d \neq 1$. In the case $d=1$, we have to separately check that \eqref{eq:Gamma_Xi_0} descends to $\widetilde \Tor_0(\widehat R,P)$. Let us note that in this case, $\chi=-1$ implies $\varepsilon=-1$. We consider $\alpha \in C^0(\Sigma;\mathcal P)$ which satisfies $
    \Omega(\cdot, \int \alpha) =0$. Let $p \in P$. Then 
    \begin{equation}
\langle \Gamma,\Xi(\alpha+\Delta(p)) \rangle = \langle \gamma_0,      a^*(  \alpha + \Delta(p)) \smile^\Omega (\alpha + \Delta(p)) \rangle 
    \end{equation}
    Let us first consider the mixed terms in this expression. We have
    \begin{align}
        \langle \gamma_0 ,  a^* \Delta(p) \smile^\Omega \alpha + a^* \alpha \smile^\Omega \Delta(p) \rangle_0 &= \langle \gamma_0 ,(1-a^*)(a^* \Delta(p) \smile^\Omega \alpha) \rangle_0 \nonumber  \\
        & = \langle [S^{d-1}], a^* \Delta(p) \smile^\Omega \alpha \rangle_0  = \Omega \left(p,\int \alpha \right)_0  =0. 
    \end{align}
    Next, consider the term quadratic in $\Delta(p)$:
    \begin{equation}
        \langle \gamma_0, \Delta(p) \smile^\Omega \Delta(p) \rangle_0 = \Omega(p,p)_0=0.
    \end{equation}
    This is the only step in which we used the assumption that the scalar part of $\Omega$ is alternating. We have proved that
    \begin{equation}
        \langle \gamma_0,      a^*(  \alpha + \Delta(p)) \smile^\Omega (\alpha + \Delta(p)) \rangle_0 = \langle \gamma_0 , a^* \alpha \smile^\Omega \alpha \rangle_0,
    \end{equation}
    and therefore $\langle \Gamma, \Xi \big(\alpha+\Delta(p) \big)\rangle_0 = \langle \Gamma, \Xi(\alpha) \rangle_0$. This completes the proof that \eqref{eq:Gamma_Xi_0} defines a~function $q \colon \widetilde \Tor_{\frac{d-1}{2}}(\widehat R,P) \longrightarrow k$, including the case $d=1$. The property $q(mu)=m^2 q(u)$ is obvious, and $q(u+v)-q(u)-q(v)=b(u,v)$ follows from the cochain-level identity \eqref{eq:xi_refinement_identity_self_pair}.
\end{proof}

\subsection{Construction for Lagrangian submodules} \label{sec:quadratic_Lagrangian}
We consider a duality pair $(Q,L)$ as in Section \ref{sec:Lagrangian}. We~will assume that $d$ is even throughout. Then the form $\operatorname{Br}_p^L$ in the middle degree $p = \frac{d}{2}$ becomes a self-pairing of the module $A = \widetilde \Tor_{\frac{d}{2}}(\widehat R,Q)$. We study the $k$-bilinear form $b^L$ given by the scalar part of $\operatorname{Br}_{\frac{d}{2}}^L$. Depending on the sign $\chi = \varepsilon (-1)^{\frac{d}{2}}$, this bilinear form is either alternating or symmetric with a~distinguished quadratic refinement~$q^L$.

\begin{defn} \label{def:Lagrangian_xi}
For $\alpha \in C^{\frac{d-2}{2}}(\Sigma;\mathcal P)$ we set
\begin{align}
 \Xi(\alpha) &= (\xi_0(\alpha),\xi_1(\alpha),\dots), \\
\xi_i(\alpha) &= a^* \alpha \smile_i^\Omega \delta \alpha + \varepsilon a^* (a^* \alpha \smile_{i-1}^\Omega \alpha).
\end{align}
\end{defn}

\begin{defn}
   Recall the short exact sequence \eqref{eq:Lagrangian-short-exact-sequence}. It induces a short exact sequence of complexes
   \(
    0 \to C^\bullet(\Sigma;\mathcal L) \longrightarrow C^\bullet(\Sigma;\mathcal P) \longrightarrow C^\bullet(\Sigma ; \mathcal Q) \to 0.   
   \)
   
    Let $\alpha \in C^p(\Sigma;\mathcal P)$. We say that $\alpha$ is $L$-valued if $\alpha \in C^p(\Sigma;\mathcal L)$. We call $\alpha$ a~cocycle mod $L$ if $\delta \alpha $ is $L$-valued. We~call $\alpha$ a coboundary mod $L$ if $\alpha - \delta \theta$ is $L$-valued for some cochain $\theta $. These two conditions are equivalent to $\alpha$ being a lift of a cocycle (respectively, a coboundary) in~$C^p(\Sigma;\mathcal Q)$.  
\end{defn}

\begin{lem} \label{lem:Lagrangian_cup_i}
    If $\alpha,\beta$ are $L$-valued cochains, then $\alpha \smile_i^\Omega \beta =0$ for all $i$. 
\end{lem}

\begin{lem} \label{lem:Lagrangian-Xi}
    Let $\alpha, \beta \in C^{\frac{d-2}{2}}(\Sigma;\mathcal P)$ and set $\chi = \varepsilon (-1)^{\frac{d}{2}}$. Viewing $\Xi(\alpha), \Xi(\beta)$ as cochains in $C^{d-1}_{C_2}(\Sigma;\mathcal R^\chi)$, the following statements hold.
    \begin{enumerate}
        \item If $\alpha$ is a cocycle mod $L$ and $\beta$ is $L$-valued, then $\Xi(\alpha+\beta)-\Xi(\alpha)$ is a $\delta_\Tot$-coboundary.
        \item If $\alpha$ is a cocycle mod $L$, then $\Xi(\alpha)$ is a~$\delta_\Tot$-cocycle. If $\alpha$ is a coboundary mod $L$, then $\Xi(\alpha)$ is a~$\delta_\Tot$-coboundary.
        \item If $\alpha$ and $\beta$ are a cocycle and a coboundary mod $L$, respectively, then $\Xi(\alpha+\beta)-\Xi(\alpha)$ is a $\delta_\Tot$-coboundary. 
        \item If $\alpha$ is a cocycle mod $L$, then
        \begin{equation}
            \int a^* \alpha \smile^\Omega \delta \alpha = \langle \Gamma, \Xi( \alpha) \rangle  + \chi \overline{\langle \Gamma,\Xi(\alpha) \rangle}.
        \label{eq:Lagrangian-divide-by-two}
        \end{equation}
        \item If $\chi =1$ and $\alpha,\beta$ are cocycles mod $L$, then
        \begin{equation}
            \left( \int a^*\alpha \smile^\Omega \delta \beta \right)_0 = \langle \Gamma , \Xi(\alpha+\beta) - \Xi(\alpha)-\Xi(\beta) \rangle_0.
        \end{equation}
    \end{enumerate}
\end{lem}
\begin{proof}
    (1) Using \eqref{eq:P_valued_Steenrod} repeatedly we find that $\Xi(\alpha+\beta)-\Xi(\alpha) = \delta_\Tot (\phi_0,\phi_1,\dots )$, where
    \begin{equation}
        \phi_i = (-1)^{\frac{d-2}{2}} a^* \alpha \smile_{i}^\Omega \beta.
    \end{equation}
    
    (2) If $\alpha$ is a cocycle mod $L$, then using \eqref{eq:P_valued_Steenrod} and Lemma \ref{lem:Lagrangian_cup_i} we verify the relation
    \begin{equation}
    \delta \xi_i (\alpha) = ((-1)^{i+1} - \chi a^*) \xi_{i-1}(\alpha).
    \end{equation}
    This is equivalent to $\Xi(\alpha)$ being a total cocycle. Next, suppose that $\alpha$ is a~coboundary mod $L$. By (1), to show that $\Xi(\alpha)$ is a $\delta_\Tot$-coboundary it is enough to consider $\alpha$ which is a coboundary (rather than just a coboundary mod~$L$). Let $\alpha = \delta \theta$. We verify that 
    \begin{align}
        \xi_i(\delta \theta)& = \delta \phi_i +(-1)^{i-1} D^{i-1}_\chi \phi_{i-1}, \\
        \phi_i & = \varepsilon a^* (a^* \theta \smile_{i-1}^\Omega \delta \theta) + a^* \theta \smile_{i-2}^\Omega \theta, \nonumber
    \end{align}
    which amounts to $\Xi(\alpha) = \delta_\Tot (\phi_0,\phi_1,\dots)$. 

    (3) By (1) we can assume that $\beta = \delta \theta$. By (2), it is enough to show that $\Xi(\alpha+\beta)-\Xi(\alpha)-\Xi(\beta)$ is a coboundary. We~find that 
    $\Xi(\alpha+\delta \theta)-\Xi(\alpha)-\Xi(\delta \theta) = \delta_\Tot (\phi_0,\phi_1,\dots)$, where
    \begin{equation}
        \phi_i = a^* \theta \smile_i^\Omega \delta \alpha + \varepsilon a^* ( a^* \theta \smile_{i-1}^\Omega \alpha + (-1)^{\frac{d-2}{2}} a^* \alpha \smile_{i-1}^\Omega \theta).
    \end{equation}

    (4) and (5) are proved as for the corresponding statements in Lemma \ref{lem:self-pairing-Xi}.
\end{proof}

\begin{proof}[Proof of Theorem \ref{thm:intro-quadratic} (b)]
Set $r=\frac{d-2}{2}$. Recall that the self-pairing on
\(
\operatorname{Br}^{L}_{\frac d2}
\)
is the composition of the cup-product pairing for the duality pair
\((Q,L)\) with the connecting homomorphism. If~$u,v \in A$ are represented by cocycles in
\(C^r(\Sigma;\mathcal Q)\) with lifts
\(\alpha,\beta\in C^r(\Sigma;\mathcal P)\), then 
\begin{equation}
    b^L(u,v)=
    \left(\int a^*\alpha\smile^\Omega\delta\beta\right)_0.
    \label{eq:Lagrangian-middle-scalar-pairing}
\end{equation}
In particular, the role of $P$-valued cocycles in part~(a) is now
played by cocycles mod~$L$.

Let \(\alpha\in C^r(\Sigma;\mathcal P)\) be a cocycle mod~$L$.  Taking the
scalar part of \eqref{eq:Lagrangian-divide-by-two}, we obtain
\begin{equation}
    \left(\int a^*\alpha\smile^\Omega\delta\alpha\right)_0
    = (1+\chi)\langle\Gamma,\Xi(\alpha)\rangle_0.
    \label{eq:Lagrangian-scalar-divide-by-two}
\end{equation}
If \(\chi=-1\), this shows that \(b^L\) is alternating.

Suppose now that \(\chi=1\).  For $u \in A$ represented by $\alpha$, a cocycle mod~$L$, set
\begin{equation}
    q^L(u)=\langle\Gamma,\Xi(\alpha)\rangle_0.
    \label{eq:Lagrangian-quadratic-refinement}
\end{equation}
We first verify that this expression is well-defined in terms of $u$. Changing the lift of a fixed $Q$-valued cocycle
amounts to adding an $L$-valued cochain; invariance under this change follows
from Lemma~\ref{lem:Lagrangian-Xi} (1) and \eqref{eq:tot_pairing}.  Changing the
$Q$-valued cocycle by a coboundary amounts to changing its lift $\alpha$ by a
coboundary mod~$L$; invariance under this change follows from
Lemma~\ref{lem:Lagrangian-Xi} (3) and \eqref{eq:tot_pairing}.  Consequently,
if \(d \neq 2\), formula \eqref{eq:Lagrangian-quadratic-refinement} defines a
function
\begin{equation}
    q^L\colon
    H^r(\Sigma;\mathcal Q)
    \cong
    \widetilde\Tor_{\frac d2}(\widehat R,Q)
    \longrightarrow k.
\end{equation}

It remains to discuss the case \(d=2\). Then \(r=0\), and
\begin{equation}
    \widetilde\Tor_1(\widehat R,Q)
    \cong
    H^0(\Sigma;\mathcal Q)/\operatorname{im}(\Delta).
\end{equation}
Let \(p\in P\).  Since \(\delta\Delta(p)=0\),
Lemma~\ref{lem:Lagrangian-Xi} (5) gives
\begin{equation}
    \langle\Gamma,\Xi(\alpha+\Delta(p))\rangle_0
      =\langle\Gamma,\Xi(\alpha)\rangle_0
      +\langle\Gamma,\Xi(\Delta(p))\rangle_0.
\end{equation}
The only possibly nonzero component of \(\Xi(\Delta(p))\) is
\begin{equation}
    \xi_1(\Delta(p))
    =\varepsilon a^*\bigl(a^*\Delta(p)\smile^\Omega\Delta(p)\bigr).
\end{equation}
Its scalar part vanishes because \(\Omega(p,p)_0=0\). This is the only place where
the additional alternation assumption in dimension two is used.  It follows
that \eqref{eq:Lagrangian-quadratic-refinement} is invariant under the replacement $\alpha \mapsto \alpha + \Delta(p)$, and therefore descends to
\(\widetilde\Tor_1(\widehat R,Q)\).

Finally, Lemma~\ref{lem:Lagrangian-Xi}(5) and
\eqref{eq:Lagrangian-middle-scalar-pairing} give
\begin{equation}
    q^L(u+v)-q^L(u)-q^L(v)=b^L(u,v).
\end{equation}
Moreover, the definition of \(\Xi\) is quadratic in~$\alpha$, and hence
\(
    q^L(mu)=m^2q^L(u)
\)
for every \(m\in k\).  The symmetry of \(b^L\) follows either from this identity or directly from Corollary~\ref{cor:cup_symmetry}.
\end{proof}

\begin{exmp}[T-junction]
Let us assume that $d=2$ and the scalar part of $\Omega$ is alternating. We use the notation of Example \ref{exmp:2d_resol}. We can choose the chain $\Gamma=(\gamma_0,\gamma_1)$ given by
    \begin{equation}
        \gamma_0 = \sigma_{x,y} - \sigma_{\overline x,y}, \qquad \gamma_1 = \sigma_x.
    \end{equation}
Now let $\alpha \in C^0(\Sigma;\mathcal P)$. We find that
\begin{align}
    \langle \Gamma, \Xi(\alpha) \rangle_0 &= \Omega(\alpha_{\overline x},\alpha_y - \alpha_x)_0 - \Omega(\alpha_x,\alpha_y-\alpha_{\overline x})_0 - \Omega(\alpha_x,\alpha_{\overline x})_0 \\
    & = \Omega(\alpha_{\overline x},\alpha_y)_0 + \Omega(\alpha_y,\alpha_x)_0 + \Omega(\alpha_x,\alpha_{\overline x})_0  . \nonumber
\end{align}
In the context of Pauli stabilizer codes, if $\alpha$ is a cocycle mod $L$, it represents an anyon type. The above expression for $\langle \Gamma, \Xi(\alpha) \rangle_0$ reproduces the so-called T-junction formula \cite{levin2003fermions} for the topological spin of the corresponding anyon type.
\end{exmp}

\appendix

\section{Edge maps for an exact double complex} \label{app:double_complex}
Let $C_{\bullet,\bullet}$ be a double complex with horizontal and vertical differentials
\begin{equation}
d_h:C_{p,q}\longrightarrow C_{p-1,q},
\qquad
d_v:C_{p,q}\longrightarrow C_{p,q-1}.
\end{equation}
We assume that the two differentials commute:
\begin{equation}
d_h d_v=d_v d_h.
\end{equation}
Suppose that $C_{\bullet,\bullet}$ is concentrated in the vertical degrees $0\leq q\leq r$, with no boundedness assumption in the horizontal direction.

We denote the bottom and top rows by
\begin{equation}
A_\bullet=C_{\bullet,0},
\qquad
B_\bullet=C_{\bullet,r},
\end{equation}
respectively. We assume that every (vertical) column is exact. In other words, for each $p$ the sequence
\begin{equation}
0\longrightarrow B_p
\longrightarrow C_{p,r-1}
\longrightarrow \cdots
\longrightarrow C_{p,1}
\longrightarrow A_p
\longrightarrow 0
\end{equation}
is exact.

The purpose of this appendix is to describe four natural maps associated with the double complex. Two of them are standard inclusion and projection maps, which are not quasi-isomorphisms in general. The other two are connecting maps, which are quasi-isomorphisms owing to the exactness of the columns. Together, these maps combine to give a useful edge-to-edge map $\kappa$. We discuss conditions under which $\kappa$ becomes an isomorphism. 

\subsubsection*{The total complex}

The total complex is defined by
\begin{equation}
\operatorname{Tot}_n(C)
=
\bigoplus_{p+q=n}C_{p,q}.
\end{equation}
The differential $D$ of the total complex is defined, for $x\in C_{p,q}$, by
\begin{equation}
D(x)=(-1)^q d_h x+d_v x.
\label{eq:total_differential}
\end{equation}
The commutativity of $d_h$ and $d_v$ implies that $D^2=0$.

For a chain complex $K_\bullet$ and an integer $s$, we use the shift convention
\begin{equation}
K[s]_n=K_{n-s},
\end{equation}
with differential
\begin{equation}
d_{K[s]}=(-1)^s d_K.
\label{eq:shifted_differential}
\end{equation}

We will consider two truncations of the double complex. Let $C_{<r}$ denote the double complex obtained by deleting the top row:
\begin{equation}
(C_{<r})_{p,q}
=
\begin{cases}
C_{p,q}, & 0\leq q<r,\\
0, & q=r.
\end{cases}
\end{equation}
Similarly, let $C_{>0}$ be the double complex obtained by deleting the bottom row, with the vertical differential out of the row $q=1$ set equal to zero.

These truncations fit into short exact sequences of chain complexes
\begin{equation}\label{eq: append_SES_1}
0\longrightarrow A
\longrightarrow \operatorname{Tot}(C)
\longrightarrow \operatorname{Tot}(C_{>0})
\longrightarrow 0
\end{equation}
and
\begin{equation}\label{eq: append_SES_2}
0\longrightarrow \operatorname{Tot}(C_{<r})
\longrightarrow \operatorname{Tot}(C)
\longrightarrow B[r]
\longrightarrow 0.
\end{equation}

By~\cite[Lemma 2.7.3]{Weibel}, the exactness of the columns implies that the total complex is acyclic:
\begin{equation}
H_\bullet(\operatorname{Tot}(C))=0.
\end{equation}

\subsubsection*{The inclusion of the bottom row}

The first map we discuss is the evident inclusion
\begin{equation}
\iota_A:A\longrightarrow \operatorname{Tot}(C_{<r}).
\end{equation}
For $a\in A_n=C_{n,0}$, it is given by
\begin{equation}
\iota_A(a)=(a,0,\ldots,0) \in \bigoplus_{q=0}^{r-1} C_{n-q,q}.
\end{equation}
This is a chain map, since $a$ lies in the bottom row, where $d_v a = 0$, giving
\begin{equation}
D\iota_A(a)=d_h a=\iota_A(d_h a).
\end{equation}

In general, $\iota_A$ need not be a quasi-isomorphism. Nevertheless, 
\begin{lem} \label{lem:iota_A}
    If every row of $C_{\bullet,\bullet}$ other than the top and bottom rows is exact, $\iota_A$ is a quasi-isomorphism. 
\end{lem}

\subsubsection*{The projection onto the top row}

The second map we need is the projection
\begin{equation}
\pi_B:
\operatorname{Tot}(C_{>0})
\longrightarrow
B[r].
\end{equation}
An element $x\in\operatorname{Tot}_n(C_{>0})$ has a unique decomposition
\begin{equation}
x=x_1+x_2+\cdots+x_r,
\end{equation}
where
$x_q\in C_{n-q,q}.$
The projection is defined by
\begin{equation}
\pi_B(x)=x_r.
\end{equation}

It is a chain map. The top-row component of $D x$ is
\(
(-1)^r d_h x_r.
\)
This agrees with the differential on $B[r]$, which is $(-1)^r d_h$. Hence
\begin{equation}
\pi_B D=d_{B[r]}\pi_B.
\end{equation}

We note that in general $\pi_B$ is not a quasi-isomorphism. Nevertheless, 
\begin{lem} \label{lem:pi_B}
    If every row of $C_{\bullet,\bullet}$ other than the top and bottom rows is exact, $\pi_B$ is a quasi-isomorphism. 
\end{lem}

\subsubsection*{The connecting map to the bottom row}

The short exact sequence~\eqref{eq: append_SES_1} induces a connecting map
\begin{equation}
\partial_A:
\operatorname{Tot}(C_{>0})
\longrightarrow
A[1].
\end{equation}
Let us review its chain-level description. 

For
\(
x=x_1+x_2+\cdots+x_r
\)
with $x_q\in C_{n-q,q}$, define
\begin{equation}
\partial_A(x)=d_v x_1.
\end{equation}
Since
\(
d_v x_1\in C_{n-1,0}=A_{n-1}=A[1]_n,
\)
this defines a degree-zero map to the shifted complex $A[1]$.

\begin{lem} \label{lem:edge_A}
The map
\begin{equation}
\partial_A:
\operatorname{Tot}(C_{>0})
\longrightarrow
A[1]
\end{equation}
is a quasi-isomorphism.
\end{lem}

\subsubsection*{The connecting map from the top row}

The short exact sequence~\eqref{eq: append_SES_2} gives a connecting map
\begin{equation}
\partial_B:
B[r]
\longrightarrow
\operatorname{Tot}(C_{<r})[1].
\end{equation}
For $b\in B[r]_n=B_{n-r}=C_{n-r,r}$, define
\begin{equation}
\partial_B(b)=d_v b.
\end{equation}
The element $d_v b$ belongs to $C_{n-r,r-1}$, which has total degree $n-1$. Therefore
\begin{equation}
d_v b\in \operatorname{Tot}(C_{<r})_{n-1}
=
\operatorname{Tot}(C_{<r})[1]_n.
\end{equation}

\begin{lem} \label{lem:edge_B}
    The map
\begin{equation}
\partial_B:
B[r]
\longrightarrow
\operatorname{Tot}(C_{<r})[1]
\end{equation}
is a quasi-isomorphism.
\end{lem}

\subsubsection*{The edge-to-edge map}

\begin{prop} \label{prop:edge-to-edge}
The four maps above induce homomorphisms
\begin{equation}
\kappa_n:H_n(A)\longrightarrow H_{n-r+1}(B).
\end{equation}
Equivalently, they determine a morphism
\begin{equation}
\kappa:A\longrightarrow B[r-1]
\end{equation}
in the derived category.
Moreover, $\kappa_n$ is an isomorphism if every row of $C_{\bullet,\bullet}$ other than the top and bottom rows is exact. 
\end{prop}
\begin{proof}
There are two equivalent ways to define the map. First, using the truncation obtained by deleting the top row, set
\begin{equation}
\kappa_n
=
\left(\partial_B[-1]\right)_*^{-1}
\circ
(\iota_A)_*.
\label{eq:kappa1}
\end{equation}
Here subscripts $*$ indicate maps induced on homology, and $\left(\partial_B[-1]\right)_*^{-1}$ is defined by Lemma \ref{lem:edge_B}. Therefore, $\kappa_n$ is the composition
\begin{equation}
H_n(A)
\longrightarrow
H_n(\operatorname{Tot}(C_{<r}))
\longrightarrow
H_n(B[r-1])
=
H_{n-r+1}(B).
\end{equation}
If the rows of $C_{\bullet, \bullet}$ other than the top and bottom are exact, then $\kappa_n$ is an isomorphism by Lemmas \ref{lem:iota_A} and \ref{lem:edge_B}.

Alternatively, using the truncation obtained by deleting the bottom row, set
\begin{equation}
\kappa_n
=
(\pi_B[-1])_*
\circ
\left(\partial_A[-1]\right)_*^{-1},
\label{eq:kappa2}
\end{equation}
and invoke Lemmas \ref{lem:pi_B} and \ref{lem:edge_A}.

To verify that the two definitions of $\kappa_n$ agree, it is enough to show that the
composite chain maps
\begin{equation}\label{eq:edge-f-g}
f := \iota_A \circ \partial_A[-1],
\qquad
g := \partial_B[-1] \circ \pi_B[-1],
\end{equation}
both mapping $\operatorname{Tot}(C_{>0})[-1] \to \operatorname{Tot}(C_{<r})$, are chain homotopic.

We define the degree-one map
\begin{align}
h \colon \operatorname{Tot}(C_{>0})[-1]_n  \longrightarrow
\operatorname{Tot}(C_{<r})_{n+1}, \qquad 
x_1 + \cdots + x_r  \longmapsto x_1 + \cdots + x_{r-1}.
\end{align}
Write $D_{<r}$ and
$D_{>0}$ for the total differentials of the two truncations. Since the
differential of the shifted complex $\operatorname{Tot}(C_{>0})[-1]$ is
$-D_{>0}$ by our shift convention \eqref{eq:shifted_differential}, the homotopy identity to be verified
takes the form
\begin{equation}\label{eq:edge-homotopy}
f - g \;=\; D_{<r} \circ h \;-\; h \circ D_{>0}.
\end{equation}
Passing to homology gives
$f_* = g_*$; since the connecting maps $\partial_A$ and $\partial_B$ are
quasi-isomorphisms, this is equivalent to
\begin{equation}\label{eq:edge-agree}
(\partial_B[-1])_*^{-1} \circ (\iota_A)_*
\;=\;
(\pi_B[-1])_* \circ (\partial_A[-1])_*^{-1}.
\end{equation}
\end{proof}

Let $[\alpha] \in H_n(A)$ be represented by a cycle $\alpha \in A_n$. To compute $\kappa([\alpha])$, it~suffices to find elements $x_i \in C_{n-i+1,i}$ ($1 \le i \le r$) satisfying:
\begin{equation}
    d_v x_1 = \alpha, \qquad (-1)^i d_h x_i + d_v x_{i+1} =0 \quad  \text{for} \quad   1 \leq i <r .
\end{equation}
Then $x_r$ satisfies $d_h x_r=0$, and $\kappa([\alpha]) = [x_r] \in H_{n-r+1}(B)$. If $r=2$, $\kappa$~is the connecting homomorphism in the long exact sequence of homology groups induced by a short exact sequence of chain complexes.

\section{Supports and dimensions of character modules} \label{sec:dual_dim}

\begin{lem} \label{lem:Rhat_injective_decomp}
The decomposition of $\widehat R$ into indecomposable injectives
\begin{equation} \widehat R \cong \bigoplus_{ \mathfrak p \in \mathrm{Spec}(R)} E(R / \mathfrak p)^{\oplus \mu(\mathfrak p)} \label{eq:injective_decomp} \end{equation}
has multiplicities given by
\begin{equation}
    \mu(\mathfrak p) = \begin{cases}
        1 , & \text{if } \mathfrak p \text{ is a maximal ideal}, \\
        2^{\aleph_0}, & \text{otherwise}.
    \end{cases}
\end{equation}
In particular, $\mu(\mathfrak p) \neq 0$ for all $\mathfrak p \in \mathrm{Spec}(R)$, and $\mathrm{Ass}(\widehat R) = \mathrm{Spec}(R)$.
\end{lem}
\begin{proof}
The theory of injective modules over Noetherian rings implies the existence of decomposition \eqref{eq:injective_decomp}, and the formula
\begin{equation}
    \mu(\mathfrak p) = \dim_{\kappa(\mathfrak p) } \Hom_R(\kappa(\mathfrak p),\widehat R) = \dim_{\kappa(\mathfrak p)} \mathrm{Hom}_{k}(\kappa(\mathfrak p), k),
\end{equation}
where $\kappa(\mathfrak p) = R_{\mathfrak p}/ \mathfrak p R_{\mathfrak p}$ is the field of fractions of $R / \mathfrak p$. If $\mathfrak p$ is a maximal ideal, then $\kappa(\mathfrak p)$ is a~finite field, and the character module $\mathrm{Hom}_{k}(\kappa(\mathfrak p), k)$ has the same number of elements as $\kappa(\mathfrak p)$. Therefore, $\dim_{\kappa(\mathfrak p)} \mathrm{Hom}_{k}(\kappa(\mathfrak p), k)=1$. If~$\mathfrak p$ is not maximal, then $\kappa(\mathfrak p)$ has cardinality $\aleph_0$ and its character module has cardinality $2^{\aleph_0}$. Therefore, $\mu(\mathfrak p)= 2^{\aleph_0}$. 

Finally, we have $\mathrm{Ass}(\widehat R) = \{ \mathfrak p \in \mathrm{Spec}(R) \, | \, \mu(\mathfrak p ) \neq 0 \} = \mathrm{Spec}(R)$.
\end{proof}

\begin{lem} \label{lem:Krull_equality}
    If $M$ is a finitely generated $R$-module, then $M$ and $M^{\#}$ have the same annihilators, supports, and Krull dimensions.
\end{lem}
We highlight that $M^{\#}$ need not be finitely generated over $R$.
\begin{proof}
    The inclusion $\operatorname{Ann} M \subset \operatorname{Ann} M^{\#}$ is clear, and the opposite inclusion holds because the functor $(-)^{\#}$ is faithful. 

    If $N$ is an injective $R$-module, then $\operatorname{supp} \Hom(M,N) = \operatorname{supp} M \cap \operatorname{Ass} N$. In~the case $N = \widehat R$ this gives $ \operatorname{supp} M^{\#} = \operatorname{supp} M$ by Lemma \ref{lem:Rhat_injective_decomp}. Since $M$ and $M^{\#}$ have equal supports, their Krull dimensions are equal. 
\end{proof}

The following equivalence can be concluded from Lemma \ref{lem:Krull_equality} and standard properties of character groups.

\begin{lem} \label{lem:dual_dim} 
Let $M$ be a finitely generated $R$-module. The following conditions are equivalent.
\begin{enumerate}
    \item $M$ is finite.
    \item $M$ has finite length. 
    \item $M^\#$ is finite.
    \item $M^{\#}$ has finite length.
    \item $M^{\#}$ is countable.
    \item $M^{\#}$ is finitely generated over $R$. 
\end{enumerate}
\end{lem}

\vspace{0.5cm}

 \textbf{Acknowledgments.} 
 We thank Agn\'es Beaudry, Alex Bols, Pavel Etingof, Mike Hermele, Wilbur Shirley, and Evan Wickenden for discussions. The work of B. R. is supported by the National Science Centre (NCN) grant Sonata Bis 13 (project number 2023/50/E/ST1/00439). B. Y. is supported by the AMS--Simons Travel Grant. 

    \vspace{0.5cm}

    \textbf{AI Disclosure.} Language models, including GPT and Gemini, were used for assistance with writing, proofreading, and other routine tasks in the preparation of this article. We also learned the idea behind the proof of exactness of the complex in Theorem A from GPT. Other mathematical ideas and results in this article are due to the authors.

\vspace{0.5cm}   

    \textbf{Data availability statement.} No datasets were generated or analysed during the current study.

\bibliographystyle{plain}
\bibliography{bib}
\end{document}